\documentclass[10pt, letterpaper,reqno]{amsart}
\usepackage{amssymb}
\usepackage{amsmath}
\usepackage{amsfonts, color}
\usepackage{graphicx}
\usepackage{mathrsfs,amssymb, slashed, cite}

\usepackage{hyperref}
\usepackage{cleveref}

\let\Re=\undefined\DeclareMathOperator*{\Re}{Re}
\let\Im=\undefined\DeclareMathOperator*{\Im}{Im}

\newcommand{\R}{\mathbb{R}}
\newcommand{\C}{\mathbb{C}}

\newcommand{\F}{\mathcal{F}}

\newtheorem{lemma}{Lemma}[section]

\newtheorem{proposition}{Proposition}[section]

\newtheorem{theorem}{Theorem}[section]
\numberwithin{equation}{section}

\newcommand{\eps}{\varepsilon}

\newcommand{\qtq}[1]{\quad\text{#1}\quad}

\begin{document}

\title[Scattering for DMNLS]{Small and large data scattering \\ for the dispersion-managed NLS}
\author{Jumpei Kawakami}
\address{Division of Mathematics and Mathematical Sciences, Kyoto University}
\email{jumpeik@kurims.kyoto-u.ac.jp}

\author{Jason Murphy}
\address{Department of Mathematics, University of Oregon}
\email{jamu@uoregon.edu}

\maketitle

\begin{abstract} We prove several scattering results for dispersion-managed nonlinear Schr\"odinger equations.  In particular, we establish small-data scattering for both `intercritical' and `mass-subcritical' powers by suitable modifications of the standard approach via Strichartz estimates.  In addition, we prove scattering for arbitrary data in a weighted Sobolev space for intercritical powers by establishing a pseudoconformal energy estimate.  We also rule out (unmodified) scattering for sufficiently low powers.  Finally, we give some remarks concerning blowup for the focusing equation. 
\end{abstract}

\section{Introduction}

Our interest in this work is in the scattering theory for dispersion-managed nonlinear Schr\"odinger equations (NLS).  Specifically, we consider equations of the form
\begin{equation}\label{nls}
i\partial_t u + \Delta u = \int_0^1 e^{-i\sigma\Delta}\bigl[|e^{i\sigma\Delta}u|^p e^{i\sigma\Delta} u\bigr]\,d\sigma,\quad (t,x)\in\R\times\R^d.
\end{equation}
This model arises in the field of nonlinear optics, particularly in the setting of `dispersion management' and in the investigation of optical solitons.  The particular model under consideration is sometimes called the `Gabitov--Turitsyn' equation, and arises in the \emph{strong dispersion management} regime.  The relevant literature is quite extensive by now, but we would like to refer the reader to \cite{Agrawal, Kurtzke,CL, Lushnikov, GT1, HL, PZ, ZGJT, CHL, CL} for a representative sample of relevant results, and particularly point the reader to the introduction of \cite{CL} for a succinct derivation of the particular equation \eqref{nls} from the more general dispersion-managed model.  

\bigskip

\textbf{Small-data scattering.} Our first main results concern the small-data global well-posedness and scattering for \eqref{nls}.  In this regime the sign of the nonlinearity in \eqref{nls} is irrelevant, but we stick with the $+$ sign (corresponding to the defocusing case) for the sake of consistency throughout the paper.  The results may be stated as follows:

\begin{theorem}[Small-data scattering, intercritical case]\label{T1} Let $d\geq 1$ and $p\in[\frac{4}{d},\frac{4}{d-2}]$ (with $p\in[\tfrac{4}{d},\infty)$ in dimensions $d\in\{1,2\}$). Let $\varphi \in H^{s_c}(\R^d)$, where
\[
s_c=\tfrac{d}{2}-\tfrac{2}{p}\in[0,1].
\]
If $\||\nabla|^{s_c} \varphi\|_{L^2}$ is sufficiently small, then there exists a unique, global-in-time solution $u:\R\times\R^d\to\C$ to \eqref{nls} with $u|_{t=0}=\varphi$.  Moreover, $u$ \emph{scatters} in $H^{s_c}$.  That is, there exist unique $u_\pm\in H^{s_c}$ such that
\[
\lim_{t\to\pm\infty}\|u(t)-e^{it\Delta}u_\pm\|_{H^{s_c}(\R^d)}=0. 
\]
\end{theorem}

\begin{theorem}[Small-data scattering, mass-subcritical case]\label{T2} Let $d\geq 1$ and $p\in(\tfrac{2}{d},\tfrac{4}{d})\cap[\tfrac{4}{d+2},\tfrac{4}{d})$.  Let $\varphi\in \F H^{\gamma}(\R^d)$, where
\[
\gamma = \tfrac{2}{p}-\tfrac{d}{2}\in(0,1]. 
\]
If $\| |x|^\gamma\varphi\|_{L^2}$ is sufficiently small, then there exists a unique, global-in-time solution $u:\R\times \R^d\to\C$ to \eqref{nls} with $u|_{t=0}=\varphi$.  Moreover, $u$ \emph{scatters} in $\F H^{\gamma}$.  That is, there exist unique $u_\pm\in \F H^{\gamma}$ such that 
\[
\lim_{t\to\pm\infty}\|e^{-it\Delta}u(t)-u_\pm\|_{\F H^{\gamma}(\R^d)}=0. 
\]
\end{theorem}

Theorems~\ref{T1}~and~\ref{T2} parallel sharp scattering results for the standard power-type NLS (see e.g. \cite{CW}). The parameters $s_c,\gamma$ are the critical regularity/weight associated with the NLS with nonlinearity $|u|^p u$, obtained by considering the $L^2$-based spaces of data whose norms are invariant under the rescaling that preserves the class of solutions. Theorems~\ref{T1}~and~\ref{T2} give small-data scattering in the appropriate space for critical regularity/weight taking values between $0$ and $1$. The additional constraint $p>\tfrac{2}{d}$ in Theorem~\ref{T2} corresponds to the `short-range' case for the standard NLS; we are also able to prove that scattering fails below this exponent (see Theorem~\ref{T4} below).

We were inspired to consider the small-data scattering problem for \eqref{nls} by the recent work of \cite{ChoiLeeLeeScattering}.  In that work, the authors establish scattering for $p=\tfrac{4}{d}$ (the `mass-critical exponent' for NLS) in dimensions $d=1,2$ for small data in $L^2$. Their strategy was to employ the $U^p$, $V^p$ atomic spaces, observing that these spaces are compatible with the structure of the DMNLS nonlinearity.  At a technical level, the key ingredient in \cite{ChoiLeeLeeScattering} was a bilinear Strichartz estimate for linear solutions (extended to the $U^p$ space by the transference principle). 

In this paper, we demonstrate that a straightforward modification of the standard approach is sufficient to recover a suitable small-data scattering theory.  As the nonlinearity involves averaging with via the free propagator over a unit time interval, we choose function spaces that incorporate such an averaging.  Thus, instead of proving a bound of the form $u\in L_t^q L_x^r(\R\times\R^d)$, we instead prove $e^{i\theta\Delta}u \in L_\theta^\infty L_t^q L_x^r([0,1]\times\R\times\R^d)$.   To handle the effect of $e^{i\theta\Delta}$ in the estimates, we use a slight modification of the standard Strichartz estimate (inspired by \cite{Kawakami}), which we refer to throughout this paper as a `shifted Strichartz estimate' (see Proposition~\ref{P:shifted} below). 

The proof of Theorem~\ref{T1} requires estimating derivatives of the solution.  This is straightforward, as derivatives commute with the free propagator.  For Theorem~\ref{T2} (in the mass-subcritical case), we instead need to use (fractional) powers of the Galilean operator $J(t)=x+2it\nabla$ and understand how these operators commute with the free propagator and interact with the DMNLS nonlinearity.  Ultimately, we find that these operators are compatible with the structure of the nonlinearity; in fact, this was already observed in \cite{MVH}, which established a small-data modified scattering result for the $1d$ cubic problem.  We also remark that in order to work with minimal decay assumptions, we need to utilize Lorentz spaces in time (in order to deal with a weak singularity in time that stems from the use of a Sobolev embedding estimate involving the Galilean operator). 

\bigskip

\textbf{Large-data scattering.} Our next main result concerns the large-data scattering theory for \eqref{nls} in the defocusing regime.  We prove scattering for intercritical powers (including the `energy-critical' power $p=\tfrac{4}{d-2}$) for arbitrary data chosen from the weighted space $\Sigma := H^1\cap \F H^1$. We cover the full `mass-supercritical' range, except for a small gap in dimension $d=1$.  In particular, our arguments require $p>p_0:=3+\sqrt{5}\approx 5.2361$ in dimension $d=1$.

Our main result may be stated as follows:

\begin{theorem}\label{T3} Let $d\geq 1$ and $\tfrac{4}{d}\leq p\leq\tfrac{4}{d-2}$ (with $p<\infty$ in dimensions $d\in\{1,2\}$).  

For any $u_0\in\Sigma$, there exists a unique, global-in-time solution $u:\R\times\R^d\to\C$ to \eqref{nls} with $u|_{t=0}=u_0$. 

Moreover, if $p$ satisfies
\[
\begin{cases} p_0<p<\infty & d=1, \\ 2<p<\infty & d=2, \\ \tfrac{4}{d}<p\leq\tfrac{4}{d-2} & d\geq 3,\end{cases}
\]
then $u$ scatters in $\Sigma$ as $t\to\pm\infty$.  That is, there exist unique $u_\pm\in\Sigma$ such that
\[
\lim_{t\to\infty}\|e^{-it\Delta}u(t)-u_\pm\|_{\Sigma} = 0. 
\]
\end{theorem}

Global well-posedness for the range of powers considered in Theorem~\ref{T3} follows from the standard contraction mapping argument, incorporating a Strichartz estimate in the $(\sigma,x)$ variables. We find that for the full range $\tfrac{4}{d}\leq p\leq \tfrac{4}{d-2}$, we obtain a local existence time that depends only on the $H^1$-norm of the initial data (even for $p=\tfrac{4}{d-2}$).  In particular, it seems that the averaging effect of the nonlinearity is helpful for short times.  Global well-posedness then follows from the conservation of mass and energy (in the defocusing case). 

For the proof of scattering, we adapt the pseudoconformal energy estimate used for the standard NLS (see e.g. \cite{Cazenave, MurphyExpository}).  The calculations share some similarities to those carried out in \cite{ChoiHongLee}, which carried out a virial calculation for \eqref{nls} in one space dimension to establish a global existence/blow up dichotomy below a ground state threshold in the focusing case (more on this below). 

The pseudoconformal energy estimate is essentially based on calculating the time derivative of the quantity $\|Ju\|_{L^2}^2$ (where $J(t)=x+2it\nabla$ is the Galilean vector field introduced above), utilizing the conservation of energy and the Gronwall inequality.  For power-type nonlinearities, one obtains a decay estimate on the $L_x^{p+2}$-norm of the solution for any power $p>0$. These estimates imply critical space-time bounds (which then yield scattering) exactly when $p$ exceeds the `Strauss exponent' $\alpha_0(d)=\tfrac{2-d+\sqrt{d^2+12d+4}}{2d}$. We refer the reader to \cite{MurphyExpository} for more details.

In the setting of \eqref{nls}, one can also obtain an analogue of the pseudoconformal energy identity (see Lemma~\ref{PCE}), albeit with some key differences.  First, the coefficient for the main term in the derivative of the pseudoconformal energy is different, involving the factor $-[\tfrac{dp}{2}-4]$ rather than $-[\tfrac{dp}{2}-2]$.  Second, due to the structure of the nonlinearity, one also encounters `lower order' terms without a `good sign'.  As we will show, the lower order terms lead to an integrable contribution in the Gronwall argument, and thus do not change the argument in an essential way.  However, due to the different coefficient in the main term, we obtain new estimates only in the regime $p>\tfrac{4}{d}$.  

In particular, we obtain a bound on the growth of $Ju$ in $L^2$, as well as a decay estimate for the $L_{\sigma,x}^{p+2}$-norm of $e^{i\sigma\Delta}u(t)$ (see Proposition~\ref{P:decay}), both of which degenerate to the trivial bound as $p\downarrow \tfrac{4}{d}$).  Although the decay estimates we obtain are weaker than the corresponding estimates for the standard NLS, they are still sufficient to obtain critical (averaged) space-time bounds (and thus scattering) essentially for the entire mass-supercritical range.  In particular, by utilizing a Klainerman--Sobolev estimate, a Strichartz estimate in the $(\sigma,x)$ variables, and the bound on $Ju$ in $L^2$, we can obtain a critical bound of the form $e^{i\sigma\Delta}u\in L_{t,\sigma}^{q_c} L_x^{r_c}$ for $p>\tfrac{4}{d}$ in dimensions $d\geq 2$ (see Proposition~\ref{P:bounds} below).  Our approach requires some modification in dimension $d=1$, and it is here that we run into some technical obstructions leading to the constraint $p>3+\sqrt{5}$.  Scattering in $1d$ for powers $4<p\leq 3+\sqrt{5}$ remains an open problem.  It is also natural to consider the problem of large-data scattering for the `mass-critical' exponent $p=\tfrac{4}{d}$ in general dimensions $d\geq 1$.

\bigskip

\textbf{No scattering for low powers.} Our next result demonstrates that no (unmodified) scattering is possible if the power of the nonlinearity is sufficiently low.

\begin{theorem}\label{T4} Let $d\geq 1$.  Suppose $0<p\leq 1$ and $p\leq \tfrac{2}{d}$.  If $u$ is a forward-global solution to \eqref{nls} such that 
\[
\lim_{t\to\infty} \|u(t)-e^{it\Delta}\varphi\|_{L_x^2} = 0 
\]
for some $\varphi\in L^2$, then $\varphi\equiv 0$.  The analogous statement holds as $t\to-\infty$. 
\end{theorem}

Theorem~\ref{T4} parallels the corresponding result for the standard power-type NLS (see e.g. \cite{Glassey, Strauss}), and indeed we can follow the same approach as in these works.  The idea of the proof is essentially to show that the time derivative of the inner product of $u(t)$ and $e^{it\Delta}\varphi$ cannot be integrable unless $\varphi\equiv 0$, as otherwise it is of order $|t|^{-\frac{dp}{2}}$.  At a technical level, the proof requires little more than the dispersive estimate for the free Schr\"odinger equation.

Note that Theorem~\ref{T4} does not seem to cover the full range of expected powers, as the proof breaks down in the regime $1<p\leq 2$ in dimension $d=1$.  For the standard NLS in the defocusing case, this gap was closed by the work of \cite{Barab}, which imposed a weighted assumption on the solution and utilized the decay estimates provided by the pseudoconformal energy estimate.  As discussed above, however, the pseudoconformal energy estimate we obtained for \eqref{nls} does not provide any decay estimate if $p\leq \tfrac{4}{d}$.  Thus the range $1<p\leq 2$ in dimension $d=1$ remains open for \eqref{nls}.

\bigskip

\textbf{Remarks on blowup.} Finally, we would like to make some remarks concerning blowup in the focusing case related to the recent work \cite{ChoiHongLee}.  In that work, the authors considered the focusing DMNLS in one space dimension, i.e.
\begin{equation}\label{focusing}
(i\partial_t + \Delta)u + \int_0^1 e^{-i\sigma\Delta}(|e^{i\sigma\Delta}u|^p e^{i\sigma\Delta}u)\,d\sigma=0,\quad (t,x)\in\R\times\R.
\end{equation}
They defined a `ground state threshold' given in terms of the optimal constant in the global Strichartz--Gagliardo--Nirenberg inequality
\[
\|e^{i\sigma\Delta}\varphi\|_{L_{\sigma,x}^{p+2}(\R\times\R)}^{p+2} \leq C_0 \|\varphi\|_{L_x^2}^{\frac{p+8}{2}}\|\partial_x \varphi\|_{L_x^2}^{\frac{p-4}{2}},
\]
which they proved admits an optimizer $Q\in H^1$ satisfying
\[
-Q+\Delta Q + \int_\R e^{-i\sigma\Delta}(|e^{i\sigma\Delta}Q|^p e^{i\sigma\Delta}Q)\,d\sigma=0. 
\]
While uniqueness of such optimizers is not yet known, the sharp constant $C_0$ can be described purely in terms of the mass/energy of any such optimizer. Defining the conserved mass by $M(u)=\|u\|_{L^2}^2$ and the energies
\[
E_I(u) =\tfrac12\int_\R |\nabla u|^2\,dx - \int_I \int_\R |e^{i\sigma\Delta}u|^{p+2}\,dx\,d\sigma,\quad I\subset\R,
\]
the main result of \cite{ChoiHongLee} may be stated as follows:

\begin{theorem}[Main result of \cite{ChoiHongLee}]\label{TCHL} Let $p>8$. Suppose $u_0\in H^1(\R)$ satisfies
\begin{equation}\label{subthreshold}
M(u_0)^{\frac{p+8}{p-8}}E_{[0,1]}(u_0)<M(Q)^{\frac{p+8}{p-8}}E_\R(Q). 
\end{equation}
Let $u:(-T_{-},T_{+})\times\R\to\C$ denote the corresponding maximal-lifespan solution to \eqref{focusing}. 

(i) If 
\[
\|u_0\|_{L^2}^{\frac{p+8}{p-8}}\|\partial_x u_0\|_{L^2} < \|Q\|_{L^2}^{\frac{p+8}{p-8}}\|\partial_x Q\|_{L^2},
\]
then $T_\pm=\infty$ and $u\in L_t^\infty H_x^1(\R\times\R)$. 

(ii) If
\[
\|u_0\|_{L^2}^{\frac{p+8}{p-8}}\|\partial_x u_0\|_{L^2} >\|Q\|_{L^2}^{\frac{p+8}{p-8}}\|\partial_x Q\|_{L^2},
\]
then $u$ blows up backward in time, i.e. $T_{-}<\infty$. 
\end{theorem}

The proof in \cite{ChoiHongLee} was based on a virial identity for \eqref{focusing}.  Similar the case of the pseudoconformal energy estimate described above, one encounters terms that only have a `good sign' in one time direction.  It is for this reason that the blowup in Theorem~\ref{TCHL} was only obtained in one time direction.

In this section, we demonstrate how a modified version of time reversal symmetry may be applied to obtain blowup in both time directions:

\begin{theorem}\label{T5} Under the same assumptions of Theorem~\ref{TCHL}(ii), the solution $u$ blows up in both time directions, i.e. $T_\pm<\infty$. 
\end{theorem}

\begin{proof}[Proof of Theorem~\ref{T5}] We first claim that the following modified time reversal symmetry for \eqref{focusing} holds: If $u$ is a solution to \eqref{focusing}, then so is $v(t):=e^{-i\Delta}\overline{u}(-t)$. Indeed, suppose $u$ solves \eqref{focusing}.  Then
\begin{align*}
(i\partial_t + \Delta)v(t) & = e^{-i\Delta}\overline{[(i\partial_t+\Delta)u]}(-t) \\
& = -\int_0^1 e^{i(\sigma-1)\Delta}[|e^{-i\sigma\Delta}\bar u(-t)|^p e^{-i\sigma\Delta}\bar u(-t)]\,d\sigma \\
& = -\int_0^1 e^{-i(1-\sigma)\Delta}[e^{i(1-\sigma)\Delta}v(t)|^p e^{i(1-\sigma)\Delta}v(t)]\,d\sigma \\
& = -\int_0^1 e^{-i\sigma\Delta}[e^{i\sigma\Delta}v(t)|^p e^{i\sigma\Delta}v(t)]\,d\sigma
\end{align*}
by a change of variables in $\sigma$. 

We next observe that the conditions in \eqref{subthreshold} and (ii) are invariant under the modified time reversal symmetry.  Indeed, condition (ii) is stated at $t=0$ and involves only the $L^2$- and $\dot H^1$-norms of the data, which are not affected by an application of the free propagator. Thus it suffices remains to verify that with $u,v$ related as above, we have
\[
E_{[0,1]}(v(t)) = E_{[0,1]}(u(-t)). 
\]
For this, it suffices to check the nonlinear part of the energy.  The computation is similar to the one above:
\begin{align*}
\int_0^1\int_\R |e^{i\sigma\Delta}v(t)|^{p+2}\,dx\,d\sigma & = \int_0^1 \int_\R |e^{i(\sigma-1)\Delta}\bar u(-t)|^{p+2}\,dx\,d\sigma \\
& = \int_0^1 \int_\R |e^{i(1-\sigma)\Delta}u(-t)|^{p+2}\,dx\,d\sigma \\
& = \int_0^1 \int_\R |e^{i\sigma\Delta}u(-t)|^{p+2}\,dx\,d\sigma.
\end{align*}

Now suppose $u$ satisfies \eqref{subthreshold} and (ii).  Then by the result of \cite{ChoiHongLee}, $u$ blows up backward in time.  By the argument above, the time reversal $v(t)=e^{-i\Delta}\overline{u}(-t)$ also satisfies \eqref{subthreshold} and (ii).  Thus $v$ also blows up backward in time, so that $u$ blows up forward in time. \end{proof}

 \subsection{Organization of the paper}
 
 The rest of this paper is organized as follows: In Section~\ref{S:preliminaries}, we gather some preliminary material needed throughout the rest of the paper.  In particular, we prove the shifted Strichartz estimate, Proposition~\ref{P:shifted}. In Section~\ref{S:T1}, we prove Theorem~\ref{T1} (small-data scattering in the mass-supercritical case).  In Section~\ref{S:T2}, we prove Theorem~\ref{T2} (small-data scattering in the mass-subcritical case).  In Section~\ref{S:T3}, we prove Theorem~\ref{T3} (large data scattering for weighted data).  The key ingredient is the pseudoconformal energy identity (Lemma~\ref{PCE}).  Finally, in Section~\ref{S:T4}, we prove Theorem~\ref{T4} (failure of scattering for sufficiently low powers). 
 
 \subsection*{Acknowledgements} J. K. was supported by JST SPRING, Grant Number JPMJSP2110. J. M. was supported by NSF grant DMS-2350225.

\section{Preliminaries}\label{S:preliminaries}  We use the standard $\lesssim$ notation, along with standard notation for Lebesgue and mixed Lebesgue spaces. For example, we will deal with norms of the form
\[
\| F\|_{L_\theta^a L_t^b L_x^c(I\times J\times\R^d)} := \bigl\| \ \bigl\| \, \|F(\theta,t,x)\|_{L_x^c(\R^d)} \bigr\|_{L_t^b(J)} \bigr\|_{L_\theta^a(I)}. 
\]
We also use $'$ to denote H\"older dual exponents. We use $\F$ to denote the Fourier transform, i.e. $\F \varphi(\xi)=(2\pi)^{-\frac{d}{2}}\int_{\R^d} e^{-ix\cdot\xi}\varphi(x)\,dx$. Fractional derivatives are defined via the Fourier transform via $|\nabla|^s = \F^{-1} |\xi|^s \F$.  Finally, we use the Japanese bracket notation $\langle x\rangle = \sqrt{1+|x|^2}$. 

Throughout the paper, we use $\Sigma$ to denote the space $\Sigma=\{f: f\in H^1,\ xf\in L^2\}$. 

We will make use of the following fractional chain rule: 
\begin{proposition}[Fractional chain rule, \cite{ChristWeinstein}] Let $G\in C^1(\C)$ and $s\in(0,1]$.  Suppose $1<r,r_2<\infty$ and $1<r_1\leq\infty$ satisfy $\tfrac{1}{r}=\tfrac{1}{r_1}+\tfrac{1}{r_2}$.  Then
\[
\| |\nabla|^s G(u)\|_{L_x^r} \lesssim \|G'(u)\|_{L_x^{r_1}} \| |\nabla|^s u\|_{L_x^{r_2}}.
\]
\end{proposition}

It will be convenient at times to utilize the decomposition
\begin{equation}\label{dollard}
e^{it\Delta} = M(t)D(t)\F M(t),
\end{equation}
where
\[
[M(t)f](x) = e^{\frac{i|x|^2}{4t}}f(x),\quad [D(t)f](x) = (2it)^{-\frac{d}{2}}f(\tfrac{x}{2t}),
\]
and $\F$ is the Fourier transform. 

We recall the definition of the Lorentz space $L_t^{q,\alpha}(\R)$ (see e.g. \cite{Hunt1, Hunt2}): For $1\leq q<\infty$ and $1\leq\alpha\leq\infty$, the space $L_t^{q,\alpha}$ is defined via the quasi-norm
\[
\|f\|_{L_t^{q,\alpha}} := \bigl\| \lambda\,|\{ t\in\R:|f(t)|>\lambda\}|^{\frac{1}{q}}\bigr\|_{L^\alpha((0,\infty),\frac{d\lambda}{\lambda})},
\] 
where $|\cdot|$ denotes Lebesgue measure. Many standard inequalities extend naturally to Lorentz spaces.  For example, we have the H\"older inequality
\[
\|fg\|_{L_t^{q,\alpha}} \lesssim \|f\|_{L_t^{q_1,\alpha_1}}\|g\|_{L_t^{q_2,\alpha_2}}
\]
for $1\leq q,q_1,q_2<\infty$ and $1\leq\alpha,\alpha_1,\alpha_2\leq\infty$ satisfying
\[
\tfrac{1}{q}=\tfrac{1}{q_1}+\tfrac{1}{q_2}\qtq{and}\tfrac{1}{\alpha}=\tfrac{1}{\alpha_1}+\tfrac{1}{\alpha_2}
\]
(see e.g. \cite{Hunt1}).

We also have the following convolution inequality (an extension of Hardy--Littlewood--Sobolev or, more generally, Young's convolution inequality):
\[
\|f\ast g\|_{L^{q,2}}\lesssim \|f\|_{L^{q_1,\infty}}\|g\|_{L^{q_2,2}}\]
for
\[
1<q,q_1,q_2<\infty\qtq{satisfying}\tfrac{1}{q_1}+\tfrac{1}{q_2}=1+\tfrac{1}{q}. 
\]
This follows immediately from the standard Young convolution inequality together with an extension of the Marcinkiewicz interpolation theorem to Lorentz spaces due to Hunt \cite{Hunt2}.

We will make use of Strichartz estimates for the linear Schr\"odinger equation (cf. \cite{GinibreVelo, KeelTao, Strichartz}), including an extension making use of Lorentz spaces in time (see \cite{NakanishiOzawa}).

In what follows, we call $(q,r)$ a (Schr\"odinger) admissible pair if $2\leq q,r\leq\infty$, $\tfrac{2}{q}+\tfrac{d}{r}=\tfrac{d}{r}$, and $(d,q,r)\neq (2,2,\infty)$. We call the pair $(2,\tfrac{2d}{d-2})$ the endpoint admissible pair in dimensions $d\geq 3$. 

\begin{proposition}[Standard Strichartz estimate]
Let $I\ni 0$ be a time interval.  For any admissible pair, 
\[
\|e^{it\Delta}\varphi\|_{L_t^\infty L_x^2\cap L_t^{q,2} L_x^r(I\times\R^d)}\lesssim \|\varphi\|_{L_x^2}.
\]
Moreover, for any admissible pairs $(q_1,r_1)$ and $(q_2,r_2)$,
\[
\biggl\|\int_{0}^t e^{i(t-s)\Delta}F(s)\,ds\biggr\|_{L_t^\infty L_x^2\cap L_t^{q_1,2} L_x^{r_1}(I\times\R^d)} \lesssim \|F\|_{L_t^{q_2',2}L_x^{r_2',2}}.
\]
The same estimates hold with the Lorentz spaces $L_t^{q,2}$ replaced by the standard Lebesgue spaces $L_t^q$. 
\end{proposition}

We also need a slight modification of the standard Strichartz estimate, which we refer to throughout the paper as a `shifted Strichartz estimate' due to the presence of the additional time shift by the linear propgator.  This estimate is inspired by the methods in \cite{Kawakami}.

\begin{proposition}[Shifted Strichartz estimate]\label{P:shifted} Let $(q,r)$ be an admissible pair with $q\neq 2$. Then for any time interval $I\ni 0$ we have
\[
\biggl\| \int_0^t e^{i(t-s+\theta-\sigma)\Delta}F(\sigma,s)\,ds\biggr\|_{L_t^{q,2} L_x^r(I\times\R^d)} \lesssim \|F(\sigma,t)\|_{L_t^{q',2}L_x^{r'}(I\times\R^d)} 
\]
uniformly in $\theta,\sigma\in\R$.  The same result when the Lorentz space $L_t^{q,2}$ is replaced by the usual Lebesgue space $L_t^q$. 
\end{proposition}

\begin{proof} The Lorentz version is strictly stronger, so this is what we prove.  We mimic proof of non-endpoint Strichartz estimates, inserting the Lorentz-space version of Hardy--Littlewood--Sobolev when appropriate.  We estimate over $\R\times\R^d$; the restriction to $I\times\R^d$ can be achieved by replacing $F(\sigma,s)$ with $F(\sigma,s)\chi_{I}(s)$. 

For the sake of completeness, we provide the details here.  We estimate by duality, choosing $G\in L_t^{q',2}L_x^{r'}$ of norm one and using the dispersive estimate for $e^{it\Delta}$, H\"older's inequality in Lorentz space, translation invariance of Lorentz space, and Hardy--Littlewood--Sobolev in Lorentz space to estimate (for any $\theta,\sigma$):
\begin{align*}
\int_\R\int_0^t&\int |e^{i(t-s+\theta-\sigma)\Delta}F(\sigma,s)|\,|G(t,x)|\,dx\,ds\,dt \\
& \lesssim \int_\R \int_\R |t-s+\theta-\sigma|^{-\frac{d}{2}+\frac{d}{r}}\|F(\sigma,s)\|_{L_x^{r'}}\,\|G(t)\|_{L_x^{r'}}\,ds\,dt \\
& \lesssim \biggl\| \int_\R |t-s+\theta-\sigma|^{-\frac{2}{q}}\|F(\sigma,s)\|_{L_x^{r'}}\,ds \biggr\|_{L_t^{q,2}} \|G\|_{L_t^{q',2}L_x^{r'}}  \\
& \lesssim \bigl\| |t|^{-\frac{2}{q}}\ast \|F(\sigma,t)\|_{L_x^{r'}}\bigr\|_{L_t^{q,2}} \\
& \lesssim \||t|^{-\frac{2}{q}}\|_{L_t^{\frac{q}{2},\infty}}\|F(\sigma,t)\|_{L_t^{q',2} L_x^{r'}}. 
\end{align*}

\end{proof}

\section{Small-data scattering in the mass-supercritical case}\label{S:T1}

The goal of this section is to prove Theorem~\ref{T1}, which we reproduce informally as Theorem~\ref{T1_informal} below.  We suppose $p\in[\tfrac{4}{d},\tfrac{4}{d-2}]$ (with the replacement $p\in[\tfrac{4}{d},\infty)$ in dimensions $d=1,2$).  We define the critical regularity 
\[
s_c=\tfrac{d}{2}-\tfrac{2}{p}\in[0,1].
\]

We first introduce the exponents to be used through this section. We define
\[
q=p+2,\quad r=\tfrac{2d(p+2)}{2(d-2)+dp},\quad r_c=\tfrac{dp(p+2)}{4}.
\]
For the range of $p$ we consider, we have that $(q,r)$ is a (non-endpoint) admissible Strichartz pair.  The exponent $r_c$ belongs to $[r,\infty)$ and satisfies the scaling relation
\[
\tfrac{d}{r_c} = \tfrac{d}{r}-s_c. 
\]
We also take note of the scaling relations
\[
\tfrac{1}{q'}=\tfrac{p+1}{q},\quad \tfrac{1}{r'}=\tfrac{p}{r_c}+\tfrac{1}{r}.
\]

With the exponents fixed as above, we now prove the following.

\begin{theorem}\label{T1_informal} Let $\varphi\in H^{s_c}$.  If $\| |\nabla|^{s_c}\varphi\|_{L^2}$ is sufficiently small, we obtain a global solution to \eqref{nls} that scatters.
\end{theorem}

\begin{proof} Fix $\varphi\in H^{s_c}$. Define
\[
\Phi u(t) = e^{it\Delta}\varphi - i\int_0^t e^{i(t-s)\Delta}\biggl[\int_0^1 e^{-i\sigma\Delta}[|e^{i\sigma\Delta}u(s)|^p e^{i\sigma\Delta}u(s)]\,d\sigma\biggr]\,ds. 
\]
We will show that $\Phi$ is a contraction on a suitable complete Banach space.  In particular, we define
\[
X=\{u:\|e^{i\theta\Delta}u\|_{L_\theta^{\infty} L_t^q L_x^r} \leq 2C\|\varphi\|_{L^2},\quad \|e^{i\theta\Delta}|\nabla|^{s_c}u\|_{L_\theta^{\infty} L_t^q L_x^r} \leq 2C\||\nabla|^{s_c}\varphi\|_{L^2}\}
\]
with distance
\[
d(u,v) = \|e^{i\theta\Delta}[u-v]\|_{L_\theta^{\infty} L_t^q L_x^r}.
\]
All norms are taken over $(\theta,t,x)\in[0,1]\times\R\times\R^d$ unless otherwise indicated.  The constant $C$ in the definition of $X$ is meant to encode all of the implicit constants used in the estimates below.

We let $u\in X$ and first show $\Phi u\in X$.

We begin with the linear term in $\Phi u$.  Using (time) translation invariance and Strichartz estimates, we have
\begin{align*}
\|e^{i\theta\Delta}|\nabla|^{s_c} e^{it\Delta}\varphi\|_{L_\theta^{\infty} L_t^q L_x^r} & = \|e^{i(t+\theta)\Delta}|\nabla|^{s_c} \varphi\|_{L_\theta^{\infty} L_t^q L_x^r} \\
& = \|e^{it\Delta} |\nabla|^{s_c}\varphi\|_{L_\theta^{\infty} L_t^q L_x^r} \\
& \lesssim \|e^{it\Delta}|\nabla|^{s_c}\varphi\|_{L_t^q L_x^r} \lesssim \| |\nabla|^{s_c}\varphi\|_{L_x^2}.
\end{align*}
Using the shifted Strichartz estimate, the fractional chain rule, and Sobolev embedding, the nonlinear term is estimated as follows:
\begin{align*}
\biggl\|& e^{i\theta\Delta}|\nabla|^{s_c} \int_0^t e^{i(t-s)\Delta}\int_0^1 e^{-i\sigma\Delta}[|e^{i\sigma\Delta}u(s)|^p e^{i\sigma\Delta}u(s)]\,d\sigma\,ds\biggr\|_{L_\theta^{\infty} L_t^q L_x^r} \\
& = \biggl\|\int_0^1 \int_0^t e^{i(t-s+\theta-\sigma)} |\nabla|^{s_c}[|e^{i\sigma\Delta}u(s)|^pe^{i\sigma\Delta} u(s)]\,d\sigma\,ds\biggr\|_{L_\theta^{\infty} L_t^q L_x^r} \\
& \leq \int_0^1 \biggl\|\int_0^t e^{i(t-s+\theta-\sigma)\Delta} |\nabla|^{s_c}(|e^{i\sigma\Delta}u(s)|^p e^{i\sigma\Delta} u(s))\,ds\biggr\|_{L_\theta^{\infty} L_t^q L_x^r}\,d\sigma \\
& \lesssim \int_0^1 \| |\nabla|^{s_c}(|e^{i\sigma\Delta}u(t)|^p e^{i\sigma\Delta}u(t))\|_{L_t^{q'} L_x^{r'}}\,d\sigma \\
& \lesssim \int_0^1 \|e^{i\sigma\Delta}u\|_{L_t^q L_x^{r_c}}^p \|e^{i\sigma\Delta}|\nabla|^{s_c} u\|_{L_t^q L_x^r}\,d\sigma\\
& \lesssim \int_0^1 \|e^{i\sigma\Delta} |\nabla|^{s_c}u\|_{L_t^q L_x^r}^{p+1}\,d\sigma \\
& \lesssim \|e^{i\sigma\Delta} |\nabla|^{s_c} u\|_{L_\sigma^{\infty} L_t^q L_x^r}^{p+1}. 
\end{align*}
Thus
\begin{align*}
\|e^{i\theta\Delta}|\nabla|^{s_c}\Phi u\|_{L_\theta^{\infty} L_t^q L_x^r} & \leq C\||\nabla|^{s_c}\varphi\|_{L_x^2} + \tilde C \|e^{i\theta\Delta} |\nabla^{s_c} u\|_{L_\theta^{\infty} L_t^q L_x^r}^{p+1} \\
& \leq \||\nabla|^{s_c}\varphi\|_{L_x^2}[C+\tilde C(2C)^{p+1}\||\nabla|^{s_c}\varphi\|_{L_x^2}^p] \\
& \leq 2C\||\nabla|^{s_c}\varphi\|_{L_x^2}
\end{align*}
provided $\||\nabla|^{s_c}\varphi\|_{L_x^2}$ is sufficiently small.

The estimates without $|\nabla|^{s_c}$ are similar, but simpler.  One finds
\begin{align*}
\|e^{i\theta\Delta}\Phi u\|_{L_\theta^{\infty}L_t^q L_x^r} &\leq C\|\varphi\|_{L_x^2} + \tilde C \|e^{i\theta\Delta}|\nabla|^{s_c} u\|_{L_\theta^{\infty} L_t^q L_x^r}^p \|e^{i\theta\Delta} u\|_{L_\theta^{\infty} L_t^q L_x^r} \\
& \leq \|\varphi\|_{L_x^2}[C+\tilde C(2C)^{p+1}\||\nabla|^{s_c}\varphi\|_{L_x^2}^p] \\
& \leq 2C\|\varphi\|_{L_x^2}
\end{align*}
provided $\||\nabla|^{s_c}\varphi\|_{L_x^2}$ is sufficiently small.

To show the contraction property, we let $u,v\in X$ and estimate similarly to above:
\begin{align*}
\|&e^{i\theta\Delta}[\Phi u-\Phi v]\|_{L_\theta^{\infty} L_t^q L_x^r} \\
& \lesssim \biggl\|\int_0^1\int_0^t e^{i(t-s+\theta-\sigma)\Delta}[|e^{i\sigma\Delta}u(s)|^p e^{i\sigma\Delta}u(s)-|e^{i\sigma\Delta}v(s)|^p e^{i\sigma\Delta}v(s)]\,d\sigma\,ds\biggr\|_{L_\theta^{\infty} L_t^q L_x^r} \\
& \lesssim \int_0^1 \| |e^{i\sigma\Delta} u|^p e^{i\sigma\Delta}u - |e^{i\sigma\Delta}v|^p e^{i\sigma\Delta}v \|_{L_t^{q'}L_x^{r'}}\,d\sigma \\
& \lesssim \int_0^1 [\|e^{i\sigma\Delta}|\nabla|^{s_c} u\|_{L_t^q L_x^r}^p + \|e^{i\sigma\Delta}|\nabla|^{s_c}v\|_{L_t^q L_x^r}^p] \|e^{i\sigma\Delta}[u-v]\|_{L_t^q L_x^r} \,d\sigma \\
& \lesssim \bigl[\|e^{i\sigma\Delta}|\nabla|^{s_c} u\|_{L_\sigma^{\infty} L_t^q L_x^r}^p + \|e^{i\sigma\Delta}|\nabla|^{s_c}v\|_{L_\sigma^{\infty} L_t^q L_x^r}\bigr]\|e^{i\sigma\Delta}[u-v\|_{L_\sigma^{\infty} L_t^q L_x^r} \\
& \lesssim \| |\nabla|^{s_c}\varphi\|_{L_x^2}^p \|e^{i\sigma\Delta}[u-v]\|_{L_\sigma^{\infty} L_t^q L_x^r}.
\end{align*}
Thus $\Phi$ is a contraction if $\||\nabla|^{s_c}\varphi\|_{L_x^2}$ is sufficiently small.

It follows that $\Phi$ has a unique fixed point in $X$, which yields the desired global-in-time solution.

We next show that the solution satisfies the more standard estimates
\[
\|u\|_{L_t^\infty L_x^2} \lesssim \|\varphi\|_{L_x^2},\quad \||\nabla|^{s_c}u\|_{L_t^\infty L_x^2} \lesssim \||\nabla|^{s_c}\varphi\|_{L_x^2}. 
\]
The estimate with $|\nabla|^{s_c}$ is more involved, so we focus on this one.  First note that
\[
\||\nabla|^{s_c}e^{it\Delta}\varphi\|_{L_t^\infty L_x^2} = \||\nabla|^{s_c}\varphi\|_{L_x^2}.
\]
For the nonlinear term, we use unitarity of the free propagator and estimate as above: 
\begin{align*}
\biggl\| & |\nabla|^{s_c} \int_0^t e^{i(t-s)\Delta}\int_0^1 e^{-i\sigma\Delta}[|e^{i\sigma\Delta}u(s)|^p e^{i\sigma\Delta}u(s)]\,d\sigma\,ds \biggr\|_{L_t^\infty L_x^2} \\
& \leq \int_0^1 \biggl\| e^{-i\sigma\Delta} \int_0^t e^{i(t-s)\Delta} |\nabla|^{s_c} [|e^{i\sigma\Delta}u(s)|^p e^{i\sigma\Delta}u(s)]\,ds \biggr\|_{L_t^\infty L_x^2} \\
& \leq \int_0^1 \biggl\| \int_0^t e^{i(t-s)\Delta} |\nabla|^{s_c} [|e^{i\sigma\Delta}u(s)|^p e^{i\sigma\Delta}u(s)]\,ds \biggr\|_{L_t^\infty L_x^2} \,d\sigma \\
& \lesssim \int_0^1 \||\nabla|^{s_c}[|e^{i\sigma\Delta}u|^p e^{i\sigma\Delta}u]\|_{L_t^{q'}L_x^{r'}}\,d\sigma \\
& \lesssim \|e^{i\sigma\Delta} |\nabla|^{s_c} u\|_{L_\sigma^{\infty}L_t^q L_x^r}^{p+1} \\
& \lesssim \| |\nabla|^{s_c} \varphi\|_{L_x^2}^{p+1} \ll \| |\nabla|^{s_c}\varphi\|_{L_x^2}
\end{align*}
provided $\| |\nabla|^{s_c}\varphi\|_{L_x^2}$ is sufficiently small. 

The analogous estimates without $|\nabla|^{s_c}$ lead to
\[
\|u\|_{L_t^\infty L_x^2} \lesssim \|\varphi\|_{L_x^2} + \| |\nabla|^{s_c}\varphi\|_{L_x^2}^p \|\varphi\|_{L_x^2} \lesssim \|\varphi\|_{L_x^2}.
\]

Finally, let us prove scattering in $H^{s_c}$.  It suffices to prove that $e^{-it\Delta}u(t)$ is Cauchy in $H^{s_c}$ as $t\to\pm\infty$.  As usual, we focus on the $\dot H^{s_c}$-norm. Without loss of generality, we prove scattering forward in time only.

To prove that $e^{-it\Delta}u(t)$ is Cauchy in $\dot H^{s_c}$ as $t\to\infty$, we use similar estimates to the ones used to prove $u\in L_t^\infty \dot H_x^{s_c}$.  We fix $t_1<t_2$ and use the Duhamel formula to estimate
\begin{align*}
\| &|\nabla|^{s_c}[e^{-it_2\Delta}u(t_2)-e^{-it_1\Delta}u(t_1)]\|_{L_x^2} \\
& = \biggl\|\int_0^1 \int_{t_1}^{t_2} e^{-i(s+\sigma)\Delta}|\nabla|^{s_c}[|e^{i\sigma\Delta}u(s)|^p e^{i\sigma\Delta}u(s)]\,ds\,d\sigma\biggr\|_{L_x^2} \\
& \leq \int_0^1 \biggl\| \int_{t_1}^{t_2} e^{-is\Delta}|\nabla|^{s_c}[|e^{i\sigma\Delta}u(s)|^p e^{i\sigma\Delta}u(s)]\,ds\biggr\|_{L_x^2}\,d\sigma \\
& \lesssim \int_0^1 \||\nabla|^{s_c}[|e^{i\sigma\Delta}u(t)|^p e^{i\sigma\Delta}u(t)]\|_{L_t^{q'}L_x^{r'}((t_1,t_2)\times\R^d)} \,d\sigma \\
& \lesssim \| e^{i\sigma\Delta}|\nabla|^{s_c} u\|_{L_\sigma^{p+1}L_t^q L_x^r((0,1)\times(t_1,\infty)\times\R^d)}^{p+1} \to 0 \qtq{as}t_1\to\infty
\end{align*}
by the dominated convergence theorem.  It follows that $u$ scatters in $\dot H^{s_c}$, and a similar argument shows scattering in $L^2$. 
\end{proof}


\section{Small-data scattering in the mass-subcritical case}\label{S:T2}

In this section we prove Theorem~\ref{T2}, which we reproduce informally as Theorem~\ref{T2_informal} below.  We suppose $p\in (\tfrac{2}{d},\tfrac{4}{d})\cap[\tfrac{4}{d+2},\tfrac{4}{d})$.  We define the critical weight
\[
\gamma=\tfrac{2}{p}-\tfrac{d}{2}\in(0,1]
\]
and note that
\[
p\gamma<1.
\] 

We first introduce the exponents to be used throughout this section.  We define

\[
q=\tfrac{2(p+2)}{dp-2},\quad r=\tfrac{2d(p+2)}{4+d(2-p)},\quad r_c = \tfrac{dp(p+2)}{2(dp-2)}.
\]
For the range of $p$ we consider, we have that $(q,r)$ is a (non-endpoint) admissible Strichartz pair.  The exponent $r_c$ belongs to $(r,\infty)$ and satisfies the scaling relation
\[
\tfrac{d}{r_c} = \tfrac{d}{r} - \gamma 
\]

The exponents $q,r,r_c$ satisfy the following scaling relations:
\[
\tfrac{2}{q}+\tfrac{d}{r}=\tfrac{d}{2},\quad \tfrac{p+1}{q} + p\gamma = \tfrac{1}{q'},\quad \tfrac{p}{r_c}+\tfrac{1}{r}=\tfrac{1}{r'}.
\]

We will make use of the Galilean vector field $J(t)=x+2it\nabla$, which appears frequently in the scattering theory for mass-subcritical NLS.  One can verify directly that this operator satisfies the identities
\[
J(t) = e^{it\Delta} x e^{-it\Delta} = e^{\frac{i|x|^2}{4t}}(2it\nabla) e^{-\frac{i|x|^2}{4t}}.
\]
In order to work under minimal decay assumptions of the initial data, we introduce the fractional power of $J(t)$ via
\begin{equation}\label{Jgamma}
J^\gamma(t) = e^{it\Delta}|x|^\gamma e^{-it\Delta} =e^{\frac{i|x|^2}{4t}}(-4t^2\Delta)^{\frac{\gamma}{2}} e^{-\frac{i|x|^2}{4t}}. 
\end{equation}

We will utilize the following commutation identity frequently:
\[
J^\gamma(t_1)e^{it_2\Delta}=e^{it_2\Delta}J^\gamma(t_1-t_2). 
\]

The operator $J^\gamma$ satisfies chain rule and Sobolev embedding estimates.  We state them using the specific exponents we need here, but of course they hold for more general exponents satisfying the appropriate scaling relations (and avoiding $L^1$ and $L^\infty$ endpoints as needed). 

\begin{lemma}[Chain rule] The following estimate holds uniformly in $t$:
\[
\|J^\gamma(t)[|\varphi|^p \varphi]\|_{L_x^{r'}} \lesssim \|\varphi\|_{L_x^{r_c}}^p \|J^\gamma(t) \varphi\|_{L_x^{r}}.
\]
\end{lemma}

\begin{proof} Let 
\[
\tilde\varphi = e^{-\frac{i|x|^2}{4t}}\varphi.
\]
Then by the fractional chain rule, 
\begin{align*}
\|J^\gamma(t)[|\varphi|^p \varphi]\|_{L_x^{r'}} & =\|(-4t^2\Delta)^{\frac{\gamma}{2}}[|\tilde\varphi|^p \tilde \varphi]\|_{L_x^{r'}} \\
& \lesssim \|\tilde\varphi\|_{L_x^{r_c}}^p \| (-4t^2\Delta)^{\frac{\gamma}{2}}\tilde\varphi\|_{L_x^r} \\
& \lesssim \|\varphi\|_{L_x^{r_c}}^p \|J^\gamma(t)\varphi\|_{L_x^r}.
\end{align*}
\end{proof}

\begin{lemma}[Klainerman-Sobolev embedding] The following estimate holds uniformly in $t,\sigma$:  
\[
\|e^{i\sigma\Delta}\varphi\|_{L_x^{r_c}} \lesssim \ |t+\sigma|^{-\gamma} \|e^{i\sigma\Delta}J^\gamma(t)\varphi\|_{L_x^r}.
\]
\end{lemma}

\begin{proof} We estimate as follows: using Sobolev embedding and the commutation identity, 
\begin{align*}
\|e^{i\sigma\Delta}\varphi\|_{L_x^{r_c}} & = \| e^{-\frac{i|x|^2}{4(t+\sigma)}}e^{i\sigma\Delta}\varphi\|_{L_x^{r_c}} \\
&\lesssim \| |\nabla|^{\gamma} [e^{-\frac{i|x|^2}{4(t+\sigma)}}e^{i\sigma\Delta}]\varphi\|_{L_x^r} \\
& \lesssim |t+\sigma|^{-\gamma}\|(-4[t+\sigma]^2\Delta)^{\frac{\gamma}{2}} [e^{-\frac{i|x|^2}{4(t+\sigma)}}e^{i\sigma\Delta}\varphi]\|_{L_x^r}\\
& \lesssim |t+\sigma|^{-\gamma}\|J^\gamma(t+\sigma)e^{i\sigma\Delta}\varphi\|_{L_x^r} \\
& \lesssim |t+\sigma|^{-\gamma}\|e^{i\sigma\Delta}J^\gamma(t)\varphi\|_{L_x^r}.
\end{align*}
\end{proof}

With all exponents fixed above, we turn to the main theorem of this section. 

\begin{theorem}\label{T2_informal} Let $\varphi\in \F H^{\gamma}$. If $\| |x|^\gamma\varphi\|_{L^2}$ is sufficiently small, we obtain a global solution to \eqref{nls} that scatters. 
\end{theorem}

\begin{proof} Fix $\varphi \in \F H^\gamma$.  Define
\[
\Phi u(t) = e^{it\Delta} \varphi  -i\int_0^t e^{i(t-s)\Delta}\biggl[\int_0^1 e^{-i\sigma\Delta}[|e^{i\sigma\Delta}u(s)|^p e^{i\sigma\Delta}u(s)]\,d\sigma\biggr]\,ds. 
\]
We will show that $\Phi$ is a contraction on a suitable complete Banach space.  In particular, we define 
\[
X=\{u: \|e^{i\theta\Delta}u\|_{L_\theta^{\infty} L_t^{q,2}L_x^r} \leq 2C\|\varphi\|_{L^2},\quad \|e^{i\theta\Delta} J^\gamma u\|_{L_\theta^{\infty} L_t^{q,2} L_x^r} \leq 2C\||x|^\gamma\varphi\|_{L^2}\}
\]
with distance
\[
d(u,v)=\|e^{i\theta\Delta}[u-v]\|_{L_\theta^{\infty} L_t^{q,2}L_x^r}. 
\]
All norms are taken over $(\theta,t,x)\in[0,1]\times\R\times\R^d$ unless otherwise indicated.  The constant $C$ in the definition of $X$ is meant to encode all of the implicit constants in the estimates used below.  

We let $u\in X$ and aim to show $\Phi u\in X$.

We first estimate the linear term in $\Phi u$. Using time-translation invariance and Strichartz estimates, 
\begin{align*}
\|e^{i\theta\Delta}J^\gamma(t)e^{it\Delta}\varphi\|_{L_\theta^{\infty} L_t^{q,2} L_x^r}&  = \|e^{i(t+\theta)\Delta} |x|^\gamma \varphi\|_{L_\theta^{\infty} L_t^{q,2} L_x^r} \\
& = \|e^{it\Delta}|x|^\gamma\varphi\|_{L_\theta^{\infty} L_t^{q,2}L_x^r} \\
& \lesssim \|e^{it\Delta}|x|^\gamma\varphi\|_{L_t^{q,2} L_x^r}  \lesssim \| |x|^\gamma\varphi\|_{L_x^2}. 
\end{align*}

The nonlinear term is estimated using the shifted Strichartz estimate, the chain rule estimate for $J^\gamma$, the Klainerman--Sobolev embedding estimate, and H\"older's inequality for Lorentz spaces: 
\begin{align*}
\biggl\|& e^{i\theta\Delta}J^\gamma(t)\int_0^t e^{i(t-s)\Delta}\int_0^1 e^{-i\sigma\Delta}\bigl(|e^{i\sigma\Delta}u(s)|^p e^{i\sigma\Delta}u(s)\bigr)\,d\sigma\,ds\biggr\|_{L_\theta^{\infty} L_t^{q,2} L_x^{r}} \\
& = \biggl\|\int_0^1 \int_0^t e^{i(t-s+\theta-\sigma)\Delta}J^\gamma(s+\sigma)\bigl(|e^{i\sigma\Delta}u(s)|^p e^{i\sigma\Delta}u(s)\bigr)\,d\sigma\,ds \biggr\|_{L_\theta^{\infty} L_t^{q,2} L_x^{r}} \\
& \leq \int_0^1\, \biggl\|\int_0^t e^{i(t-s+\theta-\sigma)\Delta}J^\gamma(s+\sigma)\bigl(|e^{i\sigma\Delta}u(s)|^p e^{i\sigma\Delta}u(s)\bigr)\,ds\biggr\|_{L_\theta^{\infty}L_t^{q,2}L_x^r}\,d\sigma \\
& \lesssim \int_0^1 \bigl\| J^\gamma(t+\sigma)\bigl(|e^{i\sigma\Delta}u(t)|^p e^{i\sigma\Delta}u(t)\bigr)\|_{L_t^{q',2}L_x^{r'}}\,d\sigma \\
& \lesssim \int_0^1 \bigl\| \, \|e^{i\sigma\Delta}u(t)\|_{L_x^{r_c}}^p \|J^\gamma(t+\sigma)e^{i\sigma\Delta}u(t)\|_{L_x^r}\bigr\|_{L_t^{q',2}}\,d\sigma \\
& \lesssim \int_0^1 \bigl\|\ |t+\sigma|^{-p\gamma}\|J^\gamma(t+\sigma)e^{i\sigma\Delta}u(t)\|_{L_x^r}^{p+1}\bigr\|_{L_t^{q',2}}\,d\sigma \\
& \lesssim \int_0^1 \|\,|t+\sigma|^{-p\gamma}\|_{L_t^{\frac{1}{p\gamma},\infty}} \|e^{i\sigma\Delta}J^\gamma(t)u(t)\|_{L_t^{q,\infty}}^p \|e^{i\sigma\Delta}J^\gamma(t)u(t)\|_{L_t^{q,2}L_x^r}\,d\sigma \\
& \lesssim \int_0^1 \|e^{i\sigma\Delta}J^\gamma(t)u(t)\|_{L_t^{q,2}L_x^r}^{p+1}\,d\sigma \\
& \lesssim \|e^{i\sigma\Delta}J^\gamma(t)u(t)\|_{L_\sigma^{\infty}L_t^{q,2}L_x^{r}}^{p+1}
\end{align*}

Thus 
\begin{align*}
\|e^{i\theta\Delta}J^\gamma \Phi u\|_{L_\theta^{\infty} L_t^{q,2} L_x^r} & \leq  C\| |x|^\gamma\varphi\|_{L_x^2} + \tilde C \|e^{i\theta\Delta}J^\gamma u\|_{L_\theta^{\infty} L_t^{q,2} L_x^r}^{p+1} \\
& \leq \||x|^\gamma\varphi\|_{L_x^2}[C+ \tilde C(2C)^{p+1} \||x|^\gamma\varphi\|_{L_x^2}^p] \\
& \leq 2C\||x|^\gamma\varphi\|_{L_x^2}
\end{align*}
provided $\||x|^\gamma\varphi\|_{L_x^2}$ is sufficiently small. 

The estimates without $J^\gamma$ are similar, but simpler. One finds
\begin{align*}
\|e^{i\theta\Delta}\Phi u\|_{L_\theta^{\infty} L_t^{q,2} L_x^r} & \leq C\|\varphi\|_{L_x^2} + \tilde C \|e^{i\theta\Delta}J^\gamma u\|_{L_\theta^{\infty} L_t^{q,2} L_x^r}^p \|e^{i\theta\Delta}u\|_{L_\theta^{\infty} L_t^{q,2} L_x^r} \\
& \leq \|\varphi\|_{L^2}[C+ \tilde C (2C)^{p+1} \| |x|^\gamma \varphi\|_{L^2}^p] \\
& \leq 2C\|\varphi\|_{L^2}
\end{align*}
provided $\||x|^\gamma\varphi\|_{L^2}$ is sufficiently small. 

To show the contraction property, we let $u,v\in X$ and estimate similarly to above:
\begin{align*}
\|&e^{i\theta\Delta}[\Phi u - \Phi v] \|_{L_\theta^{\infty} L_t^{q,2} L_x^r} \\
& \lesssim \biggl\| \int_0^1 \int_0^t e^{i(t-s+\theta-\sigma)\Delta}\bigl[ |e^{i\sigma\Delta}u|^p e^{i\sigma\Delta}u - |e^{i\sigma\Delta} v|^p e^{i\sigma\Delta}v\bigr]\,d\sigma\,ds\biggr\|_{L_\theta^{\infty} L_t^{q,2} L_x^r} \\
& \lesssim \int_0^1 \bigl[\|e^{i\sigma\Delta}J^\gamma u\|_{L_t^{q,2} L_x^r}^p + \|e^{i\sigma\Delta}J^\gamma v\|_{L_t^{q,2}L_x^r}^p\bigr]\|e^{i\sigma\Delta}[u-v]\|_{L_t^{q,2}L_x^r} \,d\sigma \\
& \lesssim \bigl[\|e^{i\sigma\Delta}J^\gamma u\|_{L_\sigma^{\infty} L_t^{q,2}L_x^r}^p + \|e^{i\sigma\Delta}J^\gamma v\|_{L_\sigma^{\infty} L_t^{q,2}L_x^r}^p\bigr]\|e^{i\sigma\Delta}[u-v]\|_{L_\sigma^{\infty} L_t^{q,2} L_x^r} \\
& \lesssim \||x|^\gamma\varphi\|_{L_x^2}^p\|e^{i\sigma\Delta}[u-v]\|_{L_\sigma^{\infty} L_t^{q,2} L_x^r}.
\end{align*}
Thus $\Phi$ is a contraction if $\||x|^\gamma\varphi\|_{L_x^2}$ is sufficiently small. 

It follows that $\Phi$ has a unique fixed point in $X$, which yields the desired global-in-time solution.  

We next show that the solution satisfies the more standard estimates
\[
\|u\|_{L_t^\infty L_x^2}\lesssim \|\varphi\|_{L^2},\quad \|J^\gamma u\|_{L_t^\infty L_x^2} \lesssim \| |x|^\gamma\|_{L^2}. 
\]
The estimate with $J^\gamma$ is more involved, so we focus on this one. First note that
\[
\| J^\gamma e^{it\Delta} \varphi\|_{L_t^\infty L_x^2} = \|e^{it\Delta}|x|^\gamma\varphi\|_{L_t^\infty L_x^2} = \||x|^\gamma\varphi\|_{L_x^2}.
\]
For the nonlinear term,
\begin{align*}
\biggl\|& J^\gamma(t)\int_0^t e^{i(t-s)\Delta}\int_0^1 e^{-i\sigma\Delta}[|e^{i\sigma\Delta}u(s)|^p e^{i\sigma\Delta} u(s)]\,d\sigma\,ds \biggr\|_{L_t^\infty L_x^2} \\
& \leq \int_0^1 \biggl\| e^{-i\sigma\Delta}\int_0^t e^{i(t-s)\Delta}J^\gamma(s+\sigma)[|e^{i\sigma\Delta}u(s)|^p e^{i\sigma\Delta}u(s)]\,ds\biggr\|_{L_t^\infty L_x^2}  \,d\sigma \\
& \leq \int_0^1 \biggl\| \int_0^t e^{i(t-s)\Delta}J(s+\sigma)[|e^{i\sigma\Delta}u(s)|^p e^{i\sigma\Delta}u(s)]\,ds\biggr\|_{L_t^\infty L_x^2}\,d\sigma \\
& \lesssim\int_0^1 \| J^\gamma(t+\sigma)[|e^{i\sigma\Delta}u(t)|^p e^{i\sigma\Delta}u(t)]\|_{L_t^{q',2}L_x^{r'}} \,d\sigma \\
& \lesssim \| e^{i\sigma\Delta} J^\gamma u\|_{L_\sigma^{\infty} L_t^{q,2} L_x^r}^{p+1} \\
& \lesssim \| |x|^\gamma \varphi\|_{L_x^2}^{p+1} \ll \| |x|^\gamma \varphi\|_{L_x^2}
\end{align*}
provided $\||x|^\gamma \varphi\|_{L_x^2}$ is small.

The analogous estimates without $J^\gamma$ lead to
\[
\|u\|_{L_t^\infty L_x^2} \lesssim \|\varphi\|_{L_x^2} + \||x|^{\gamma} \varphi\|_{L_x^2}^p \|\varphi\|_{L_x^2} \lesssim \|\varphi\|_{L_x^2}. 
\]

Finally, let us prove scattering in $\F H^{\gamma}$, in the sense that $e^{-it\Delta} u(t)$ converges in $\F H^\gamma$ as $t\to\pm\infty$.  As usual, let us focus on the $ \F\dot H^{\gamma}$ norm.  Without loss of generality, we prove scattering forward in time only.  

Let us show that $\{e^{-it\Delta}u(t)\}$ is Cauchy in $\F \dot H^\gamma$ as $t\to\infty$. The estimates are similar to those used to prove $J^\gamma u\in L_t^\infty L_x^2$. We fix $t_1<t_2$ and use the Duhamel formula to estimate
\begin{align*}
\||x|^\gamma&[e^{-it_2\Delta}u(t_2) - e^{-it_1\Delta}u(t_1)]\|_{L_x^2} \\
& = \biggl\| \int_0^1 \int_{t_1}^{t_2} |x|^\gamma e^{-i(s+\sigma)\Delta}[|e^{i\sigma\Delta}u(s)|^p e^{i\sigma\Delta}u(s)]\,ds\,d\sigma\biggr\|_{L_x^2} \\
& = \biggl\| \int_0^1 \int_{t_1}^{t_2} e^{-i(s+\sigma)\Delta}J^\gamma(\tau+\sigma)[|e^{i\sigma\Delta}u(s)|^p e^{i\sigma\Delta}u(s)]\,ds\,d\sigma\biggr\|_{L_x^2} \\
& \leq \int_0^1 \biggl\| \int_{t_1}^{t_2} e^{-is\Delta}J^\gamma(s+\sigma)[|e^{i\sigma\Delta}u(s)|^p e^{i\sigma\Delta}u(s)]\,ds\biggr\|_{L_x^2}\,d\sigma \\
& \lesssim \int_0^1 \| J^\gamma(t+\sigma)[|e^{i\sigma\Delta}u(t)|^p e^{i\sigma\Delta}u(t)]\|_{L_t^{q',2}L_x^{r'}((t_1,t_2)\times\R^d)} \,d\sigma \\
& \lesssim \|e^{i\sigma\Delta}J^\gamma u\|_{L_\sigma^{p+1} L_t^{q,2} L_x^{r}((0,1)\times(t_1,\infty)\times\R^d)}^{p+1}\to 0 \qtq{as}t_1\to\infty
\end{align*}
by the dominated convergence theorem.  It follows that $u$ scatters in $\F \dot H^{\gamma}$, and a similar argument shows scattering in $L^2$. \end{proof}


\section{Large-data scattering for weighted data}\label{S:T3}

The goal of this section is to prove Theorem~\ref{T3}, which is broken up into Proposition~\ref{P:GWP} and Proposition~\ref{P:scatter} below.   We work in the defocusing case with powers $p$ satisfying $\tfrac{4}{d}\leq p\leq\tfrac{4}{d-2}$ (with $p<\infty$ in dimensions $d=1,2$). 

Throughout this section, we have found it convenient to introduce the notation
\[
U(\sigma) = e^{i\sigma\Delta} \qtq{and} w(\sigma)=U(\sigma) u. 
\]
We will also implicitly assume that integrals with respect to $\,dx\,d\sigma$ are taken over $\R^d\times[0,1]$. Thus, for example, in this notation we express the conserved energy of a solution as
\[
E(u) = \int \tfrac12 |\nabla u|^2\,dx + \iint \tfrac{1}{p+2}|w|^{p+2}\,dx\,d\sigma.
\]

We continue to make use of the operator
\[
J(t)= x+2it\nabla = U(t)x U(-t) = M(t) (2it\nabla) M(-t),
\]
along with its fractional powers, where we recall $M(t) = e^{\frac{i|x|^2}{4t}}$.

We begin by establishing global well-posedness.

\begin{proposition}[Global well-posedness in $\Sigma$]\label{P:GWP} Let $\tfrac{4}{d}\leq p\leq\tfrac{4}{d-2}$ (with $\tfrac{4}{d}\leq p<\infty$ in in dimensions $d\in\{1,2\}$).  For any $u_0\in \Sigma$, there exists a unique global-in-time solution $u\in C(\R;\Sigma)$ to \eqref{nls}.  The solution conserves the mass and energy, and remains uniformly bounded in $H_x^1$. 
\end{proposition}

\begin{proof} 

We will establish local well-posedness with existence time depending on the $H^1$-norm of the data.  As conservation of mass and energy imply uniform $H^1$ bounds, global well-posedness will then follow.  We will incorporate the $\Sigma$ assumption by using the vector field $J(t)$. 

We run a contraction mapping argument with the operator
\[
\Phi u = e^{it\Delta} u_0 - i\int_0^t\int_0^1 e^{i(t-s-\sigma)\Delta}[|e^{i\sigma\Delta}u(s)|^p e^{i\sigma\Delta}u(s)]\,d\sigma\,ds.
\]

We define the space
\[
X=\{u\in L_t^\infty([0,T];\Sigma): \|u\|_{L_t^\infty H_x^1}\leq C_0\|u_0\|_{H_x^1} \qtq{and}\|Ju\|_{L_t^\infty L_x^2}\leq C_0\|xu_0\|_{L_x^2}\},
\]
endowed with the $L_t^\infty L_x^2$ distance.  Here $C_0$ is a universal constant and the space-time norms are taken over $[0,T]\times\R^d$. 

In the estimates below, we will make use of the Strichartz admissible pair 
\[
(q,r) := (\tfrac{4(p+2)}{dp},p+2)
\]
and write $F(z)=|z|^p z$. 

We first check that $\Phi:X\to X$.  We begin by observing that
\[
\|e^{it\Delta}u_0\|_{L_t^\infty H_x^1} \lesssim \|u_0\|_{H_x^1} \qtq{and} \|J(t)e^{it\Delta}u_0\|_{L_t^\infty L_x^2}\lesssim \|xu_0\|_{L_x^2}. 
\]
Next, we use the Strichartz estimate, H\"older's inequality, and Sobolev embedding to obtain
\begin{equation}\label{av-gain}
\begin{aligned}
\biggl\|\nabla \int_0^t\int_0^1 &e^{i(t-s-\sigma)\Delta}F(e^{i\sigma\Delta}u)\,d\sigma\,ds\biggr\|_{L_t^\infty L_x^2} \\
& \lesssim \biggl\| \int_0^t \biggl\| \int_0^1 e^{-i\sigma\Delta} \nabla F(e^{i\sigma\Delta}u)\,d\sigma\biggr\|_{L_x^2} \,ds \biggr\|_{L_t^\infty} \\
& \lesssim \int_0^T \| \nabla F(e^{i\sigma\Delta}u(s))\|_{L_\sigma^{q'} L_x^{r'}} \,ds \\
& \lesssim \int_0^T \|e^{i\sigma\Delta}u(s)\|_{L_\sigma^\infty L_x^{r}}^p \| e^{i\sigma\Delta}\nabla u(s)\|_{L_\sigma^q L_x^r} \,ds \\
& \lesssim \int_0^T \|u(s)\|_{H_x^1}^p \|\nabla u(s)\|_{L_x^2}\,ds \\
& \lesssim T\|u\|_{L_t^\infty H_x^1}^p \|\nabla u\|_{L_t^\infty L_x^2}  \lesssim T\|u_0\|_{H_x^1}^p \|u_0\|_{H_x^1}. 
\end{aligned}
\end{equation}
Estimating similarly yields
\[
\biggl\| \int_0^t\int_0^1 e^{i(t-s-\sigma)\Delta}F(e^{i\sigma\Delta}u)\,d\sigma\,ds\biggr\|_{L_t^\infty L_x^2} 
\lesssim T\|u_0\|_{H_x^1}^p \|u\|_{L_t^\infty L_x^2}.
\]
Using the commutation properties and chain rule for $J(t)$, we can estimate similarly to obtain
\begin{align*}
\biggl\|J(t) \int_0^t\int_0^1 &e^{i(t-s-\sigma)\Delta}F(e^{i\sigma\Delta}u)\,d\sigma\,ds\biggr\|_{L_t^\infty L_x^2} \\
& \lesssim \biggl\| \int_0^t \biggl\| \int_0^1 e^{-i\sigma\Delta} J(s+\sigma) F(e^{i\sigma\Delta}u)\,d\sigma\biggr\|_{L_x^2} \,ds \biggr\|_{L_t^\infty} \\
& \lesssim \int_0^T \| J(s+\sigma) F(e^{i\sigma\Delta}u(s))\|_{L_\sigma^{q'} L_x^{r'}} \,ds \\
& \lesssim \int_0^T \|e^{i\sigma\Delta}u(s)\|_{L_\sigma^\infty L_x^{r}}^p \| e^{i\sigma\Delta}J(s) u(s)\|_{L_\sigma^q L_x^r} \,ds \\
& \lesssim \int_0^T \|u(s)\|_{H_x^1}^p \|J(s) u(s)\|_{L_x^2}\,ds \\
& \lesssim T\|u\|_{L_t^\infty H_x^1}^p \|J(s) u\|_{L_t^\infty L_x^2} \lesssim T\|u_0\|_{H_x^1}^p \| xu_0\|_{L_x^2}. 
\end{align*}

Choosing $T=T(\|u_0\|_{H_x^1})$ sufficiently small, we deduce that $\Phi:X\to X$. 

Now suppose $u,v\in X$.  Then, estimating as above, we find that
\begin{align*}
\|&\Phi(u) - \Phi(v)\|_{L_t^\infty L_x^2} \\
& \lesssim \biggl\| \int_0^t \int_0^1 e^{i(t-s-\sigma)\Delta}[F(e^{i\sigma\Delta}u)-F(e^{i\sigma\Delta}v)]\,d\sigma\,ds\biggr\|_{L_t^\infty L_x^2} \\
& \lesssim \biggl\| \int_0^t \biggl\| \int_0^1 e^{-i\sigma\Delta}[F(e^{i\sigma\Delta}u)-F(e^{i\sigma\Delta}v)]\,d\sigma\bigg\|_{L_x^2} \,ds \biggr\|_{L_t^\infty} \\
& \lesssim \int_0^T \|F(e^{i\sigma\Delta}u)-F(e^{i\sigma\Delta}v)\|_{L_\sigma^{q'} L_x^{r'}}\,ds \\
& \lesssim \int_0^T [\|e^{i\sigma\Delta}u(s)\|_{L_\sigma^\infty L_x^r}^p + \|e^{i\sigma\Delta}v(s)\|_{L_\sigma^\infty L_x^r}^p]\|e^{i\sigma\Delta}[u-v]\|_{L_\sigma^q L_x^r}\,ds \\
& \lesssim T\|u_0\|_{H_x^1}^p \|u-v\|_{L_t^\infty L_x^2},
\end{align*}
so that $\Phi$ is a contraction for $T=T(\|u_0\|_{H_x^1})$ sufficiently small.

We thus obtain a solution $u\in C([0,T];\Sigma)$ for $u_0\in \Sigma$, with time of existence depending only on $\|u_0\|_{H_x^1}$.  Using conservation of mass and energy, we have global \emph{a priori} bounds for $u$ in $H_x^1$, and hence we may extend the solution globally in time. \end{proof}
Note that in the result above, we obtain a local solution with time of existence depending only on the $H_x^1$-norm of the data, even in the `energy-critical' case $p=\tfrac{4}{d-2}$.  Evidently, the averaging effect of the nonlinearity is helpful for short times. Indeed, repeating the estimates above for an energy-critical power-type nonlinearity would lead to
\[
\biggl\|\int_0^t e^{i(t-s)\Delta}\nabla |u|^{\frac{4}{d-2}}u\,ds\biggr\|_{L_t^\infty L_x^2} \lesssim \|u\|_{L_t^\infty H_x^1}^{\frac{4}{d-2}}\|\nabla u\|_{L_t^\infty L_x^2}.
\]
This should be contrasted with \eqref{av-gain}, in which the analogous estimates gain a full power of $T$.

We turn now to the question of scattering for large data in $\Sigma$.  As in the case of the power-type NLS, we will proceed by computing a pseudoconformal energy identity, which is based on computing the time derivative of $\|Ju\|_{L^2}^2$.  From this identity, we will be able to use Gronwall's inequality to obtain estimates for the solution (at least, in the regime $p>\tfrac{4}{d}$).  These estimates will yield a `critical' space-time estimate (i.e. a bound involving space-time exponents $(q_c,r_c)$ obeying the critical scaling relation $\tfrac{2}{q_c}+\tfrac{d}{r_c}=\tfrac{2}{p}$).  Finally, we will use this space-time estimate to deduce scattering for the solution.

We remind the reader of the notation $w=w(\sigma)=U(\sigma)u$, which is meant to make the formulas more compact below.  When we write $w$, we really mean $w=w(\sigma,t,x)=e^{i\sigma\Delta}u(t,x)$.  Similarly, when we write $w(1)$, we really mean $w(1)=w(1,t,x)=e^{i\Delta}u(t,x)$.

\begin{lemma}[Pseudoconformal energy identity]\label{PCE} Let $w=w(\sigma)=U(\sigma)u$, with $u$ a solution to \eqref{nls}. The following identity holds: 
\begin{align*}
\tfrac{d}{dt}\bigl[\|Ju\|_{L_x^2}^2 + \tfrac{8 t^2}{p+2}\iint |w|^{p+2}\,dx\,d\sigma\bigr] & = -\tfrac{\frac{dp}{2}-4}{t} \tfrac{8t^2}{p+2}\iint |w|^{p+2}\,dx\,d\sigma \\
& \quad - \tfrac{\frac{dp}{2}-2}{t^2} \tfrac{8t^2}{p+2}\iint \sigma |w|^{p+2}\,dx\,d\sigma \\
& \quad - \tfrac{8(t+1)}{p+2}\int |w(1)|^{p+2}\,dx. 
\end{align*}
\end{lemma}

\begin{proof} Throughout the proof we denote partial derivatives with subscripts and employ the Einstein summation notation (i.e. repeated indices are summed).   

We expand
\[
\int |Ju|^2\,dx = \int |x|^2|u|^2 + 4t^2|\nabla u|^2 - 4t\Im(x\bar u\cdot\nabla u)\,dx.
\]
Thus
\begin{equation}\label{dtJu}
\begin{aligned}
\tfrac{d}{dt}\|Ju\|_{L_x^2}^2 & = \tfrac{d}{dt}\int |x|^2|u|^2 \,dx - 4\Im \int x\bar u \cdot\nabla u\,dx \\
& \quad- 4t\tfrac{d}{dt}\int \Im x\bar u \cdot \nabla u \,dx + 8t \int |\nabla u|^2\,dx \\
& \quad + 8t^2\tfrac{d}{dt}\int \tfrac12 |\nabla u|^2\,dx. 
\end{aligned}
\end{equation}

We first compute
\begin{align*}
\tfrac{d}{dt}\int |x|^2|u|^2\,dx & = -2\Im \int|x|^2 \bar u u_{jj}\,dx + 2\Im\int |x|^2 \bar u F\,dx,
\end{align*}
where we denote
\[
F=\int_0^1 U(-\sigma)[|w|^p w]\,d\sigma. 
\]

The quadratic terms lead to
\begin{align*}
2\Im\int 2x_j \bar u u_j \,dx = 4\Im \int x\bar u \cdot\nabla u \,dx,
\end{align*}
which leads to a cancellation in the first line of \eqref{dtJu}. 

The nonlinear term is computed as follows:
\begin{align*}
2\lambda&\Im \iint |x|^2 \bar u U(-\sigma)[|w|^p w]\,dx\,d\sigma \\
& = 2\lambda\Im \iint U(-\sigma)|x|^2 U(\sigma)\bar w\cdot |w|^p w\,dx \,d\sigma\\
& = 2\lambda\Im \iint \overline{J^2(\sigma)w}\cdot |w|^p w\,dx\,d\sigma.
\end{align*}
We expand
\[
J^2(\sigma) = x_k^2 - 4\sigma^2\partial_{kk} + 2i\sigma(x_k\partial_k + \partial_k x_k). 
\]

The contribution of the $x_k^2$ term is zero, since
\[
2\lambda \Im \iint |x|^2 \bar w\cdot |w|^p w \,dx = 0.
\]
For the contribution of the $\partial_{kk}$ term we use $\partial_\sigma U(\sigma)=i\Delta U(\sigma),$ which implies
\[
\partial_\sigma |w(\sigma)|^{p+2} = (p+2)\Im (|w|^p w \Delta \bar w). 
\]
Thus
\begin{align*}
2&\Im \iint (-4\sigma^2\Delta \bar w)\cdot |w|^p w\,dx\,d\sigma\\
& = -\tfrac{8}{p+2}\iint \sigma^2\partial_\sigma|w|^{p+2}\,dx\,d\sigma \\
& = \tfrac{16}{p+2}\iint \sigma |w|^{p+2}\,dx\,d\sigma - \tfrac{8}{p+2}\int |w(1)|^{p+2}\,dx. 
\end{align*}

For the remaining terms in $J^2(\sigma)$, we obtain
\begin{align*}
-4&\Re\iint \sigma[x_k\partial_k + \partial_k x_k]\bar w \cdot |w|^p w \,dx\,d\sigma \\
& = -8\iint \sigma x_k \Re[|w|^p w\partial_k\bar w]\,dx\,d\sigma - 4\Re\iint \sigma |w|^{p+2}\delta_{kk}\,dx\,d\sigma \\
& = -\tfrac{8}{p+2}\iint \sigma x_k \partial_k |w|^{p+2}\,dx\,d\sigma - 4d\iint \sigma|w|^{p+2}\,dx\,d\sigma \\
& = \tfrac{8d}{p+2}\iint \sigma |w|^{p+2}\,dx\,d\sigma -4d\iint \sigma |w|^{p+2}\,dx\,d\sigma. 
\end{align*}

Thus
\begin{equation}\label{dtJu1}
\begin{aligned}
\tfrac{d}{dt}& \int |x|^2 |u|^2\,dx - 4\Im\int x\bar u\cdot\nabla u \,dx \\
& = \tfrac{4(4-dp)}{p+2}\iint \sigma |w|^{p+2}\,dx\,d\sigma - \tfrac{8}{p+2}\int |w(1)|^{p+2}\,dx. \\
\end{aligned}
\end{equation}

Our next task is to compute $\tfrac{d}{dt}\int \Im x\bar u \cdot\nabla u\,dx.$ We begin by using the equation \eqref{nls} to obtain
\begin{align*}
\tfrac{d}{dt}\Im \int x_k \bar u u_k\,dx & = \Re\int x_k[\bar u u_{jjk} -\bar u_k u_{jj}]\,dx \\
& \quad + \Re\iint x_k[u_kU(\sigma)|w|^p \bar w - \bar u U(-\sigma)\partial_k(|w|^p w)]\,dx\,d\sigma.
\end{align*}
After integration by parts, the contribution of the first term becomes 
\[
\Re\int x_k(\bar u u_{jjk} - \bar u_k u_{jj})\,dx = \int 2|\nabla u|^2\,dx.
\]
We turn to the nonlinear contribution. Writing $J_k(\sigma) = x_k+2i\sigma\partial_k$, we use unitarity of $U(\sigma)$ to obtain
\begin{align*}
&\Re\iint x_k[u_k\,U(\sigma)|w|^p \bar w - \bar u U(-\sigma)\partial_k(|w|^p w)]\,dx\,d\sigma \\
& = \Re\iint \overline{U(\sigma)x_kU(-\sigma) w_k}\,|w|^p w - \overline{U(\sigma)x_k U(-\sigma) w}\,\partial_k(|w|^p w)\,dx\,d\sigma \\
& = \Re \iint |w|^p w \cdot \overline{J_k\partial_k w + \partial_k J_k w}\,dx\,d\sigma.
\end{align*}

We now compute
\[
J_k\partial_k + \partial_k J_k = x_k\partial_k + \partial_k x_k + {4}i\sigma\partial_{kk}. 
\]
The contribution of the real part of this operator is
\begin{align*}
\Re&\iint |w|^p w(x_k\partial_k \bar w + \partial_k[x_k\bar w])\,dx\,d\sigma \\
& = \Re\iint 2|w|^p w \bar w_k x_k + |w|^{p+2}\delta_{kk} \,dx\,d\sigma \\
& = \tfrac{2}{p+2}\iint \partial_k (|w|^{p+2})\,x_k\,dx\,d\sigma + d\iint |w|^{p+2}\,dx\,d\sigma \\
& = (d-\tfrac{2d}{p+2})\iint |w|^{p+2}\,dx\,d\sigma.
\end{align*}
The contribution of the imaginary part of this operator is 
\begin{align*}
{4}&\Im \iint \sigma|w|^p w\Delta\bar w\,dx\,d\sigma \\
& = \tfrac{{4}}{p+2}\iint \sigma\partial_\sigma(|w|^{p+2})\,dx\,d\sigma \\
& = -\tfrac{{4}}{p+2}\iint |w|^{p+2}\,dx\,d\sigma+\tfrac{{4}}{p+2}\int |w(1)|^{p+2}\,dx. 
\end{align*}

Thus (returning to \eqref{dtJu}),
\begin{equation}\label{dtJu2}
\begin{aligned}
-4t&\tfrac{d}{dt}\int \Im x\bar u \cdot\nabla u \,dx + 8t\int |\nabla u|^2\,dx \\
& = -4 t\bigl[\tfrac{dp-{4}}{p+2}\bigr]\iint |w|^{p+2}\,dx\,d\sigma -\tfrac{{16}t}{p+2}\int |w(1)|^{p+2}\,dx.
\end{aligned}
\end{equation}

For the last term in \eqref{dtJu}, we use conservation of energy, which implies
\[
8t^2\tfrac{d}{dt}\int \tfrac12|\nabla u|^2\,dx = -\tfrac{8t^2}{p+2}\tfrac{d}{dt}\iint |w|^{p+2}\,dx\,d\sigma. 
\]

Combining this with \eqref{dtJu1} and \eqref{dtJu2}, we find
\begin{align*}
\tfrac{d}{dt}\|Ju\|_{L_x^2}^2 & = \tfrac{4(4-dp)}{p+2}\iint \sigma |w|^{p+2}\,dx\,d\sigma + \tfrac{4 t({4}-dp)}{p+2}\iint |w|^{p+2}\,dx\,d\sigma \\
& \quad-\tfrac{8({2}t+1)}{p+2}\int |w(1)|^{p+2}\,dx \\
& \quad- \tfrac{8t^2}{p+2}\tfrac{d}{dt}\iint |w|^{p+2}\,dx\,d\sigma.
\end{align*}
Therefore
\begin{equation}\label{dte}
\begin{aligned}
\tfrac{d}{dt}\bigl[\|Ju\|_{L_x^2}^2 + \tfrac{8 t^2}{p+2}\iint |w|^{p+2}\,dx\,d\sigma\bigr] & = -\tfrac{\frac{dp}{2}-{4}}{t} \tfrac{8t^2}{p+2}\iint |w|^{p+2}\,dx\,d\sigma \\
& \quad - \tfrac{\frac{dp}{2}-2}{t^2} \tfrac{8t^2}{p+2}\iint \sigma |w|^{p+2}\,dx\,d\sigma \\
& \quad - \tfrac{8(2t+1)}{p+2}\int |w(1)|^{p+2}\,dx,
\end{aligned}
\end{equation}
as desired. 
\end{proof}

From the pseudoconformal energy identity, we will derive the following decay estimates.  We note that in practice, we only use the bound on $Ju$ in $L^2$, but for the sake of completeness we provide the bounds obtained for $e^{i\sigma\Delta}u\in L_{\sigma,x}^{p+2}$ as well. 

\begin{proposition}[Decay estimates]\label{P:decay} Let $\tfrac{4}{d}<p\leq\tfrac{4}{d-2}$ and let $u$ be the global solution to \eqref{nls} with $u|_{t=0}\in\Sigma$. The following estimates holds uniformly in $t$:
\[
\|Ju\|_{L_x^2} \lesssim \begin{cases} 1 & p > \frac{8}{d}, \\ \langle t\rangle^{2-\frac{dp}{4}} & \frac{4}{d}<p\leq \frac{8}{d}, \end{cases}
\]
and
\[
\|w\|_{L_{\sigma,x}^{p+2}} \lesssim \begin{cases} \langle t\rangle^{-\frac{2}{p+2}} & p>\frac{8}{d}, \\ \langle t\rangle^{-\frac{dp-4}{2(p+2)}} & \frac{4}{d}< p \leq \frac{8}{d}.\end{cases}
\]
\end{proposition}

\begin{proof}  We will focus on proving estimates for $t\geq 0$, although we will be careful that our approach is valid for obtaining estimates for the full interval $t\in\R$.  

By local well-posedness in $C(\R;\Sigma)$, it suffices to estimate on the interval $t\in[1,\infty)$. We will obtain the desired estimates by using the pseudoconformal energy identity and Gronwall's inequality.  In particular we will obtain an upper bound on the time derivative of the pseudoconformal energy
\[
e(t):=\|Ju\|_{L_x^2}^2 + \tfrac{8t^2}{p+2}\iint |w|^{p+2}\,dx\,d\sigma.
\]
The reader should review the identity above, see e.g. \eqref{dte}, as we will refer to the terms in this identity in what follows. 

First consider the case $p>\tfrac{8}{d}$. In this case, the terms
\[
-\tfrac{\frac{dp}{2}-4}{t} \tfrac{8t^2}{p+2}\iint |w|^{p+2}\,dx\,d\sigma \qtq{and} -\tfrac{16t}{p+2}\int |w(1)|^{p+2}\,dx
\]
both have a good sign.  That is, they are negative for $t>0$ and positive for $t<0$.  Thus, they can be discarded in what follows. 

The remaining terms on the right-hand side of the pseudoconformal energy identity \eqref{dte} actually have a good sign for $t>0$, but not for $t<0$.  Thus we must estimate them in what follows. In particular, we obtain
\[
\tfrac{d}{dt}e \lesssim \langle t\rangle ^{-2} e + \|w(1)\|_{L_x^{p+2}}^{p+2},\quad p>\tfrac{8}{d}. 
\] 

When $\tfrac{4}{d}<p\leq\tfrac{8}{d}$, the term
\[
\tfrac{4-\frac{dp}{2}}{t}\tfrac{8t^2}{p+2}\iint |w|^{p+2}\,dx\,d\sigma
\]
now has a bad sign (positive for $t>0$ and negative for $t<0$) and must be included in the differential inequality, yielding
\[
\tfrac{d}{dt} e \leq \tfrac{4-\frac{dp}{2}}{t} e + C_1\langle t\rangle^{-2}e + C_2 \|w(1)\|_{L_x^{p+2}}^{p+2},\quad \tfrac{4}{d}<p\leq\tfrac{8}{d}. 
\]

We estimate the $w(1)$ term using the Klainerman--Sobolev embedding, interpolation, and the commutation properties of $J$.  First, we estimate
\begin{align*}
\|w(1)\|_{L_x^{p+2}} &\lesssim \| \bar M(t+1)w(1)\|_{L_x^{p+2}} \\
& \lesssim \| |\nabla|^{\frac{dp}{2(p+2)}} \bar M(t+1)w(1)\|_{L_x^2} \\
& \lesssim \|u\|_{L_x^2}^{1-\frac{dp}{2(p+2)}}\|\nabla \bar M(t+1)w(1)\|_{L_x^2}^{\frac{dp}{2(p+2)}} \\
& \lesssim_u |1+t|^{-\frac{dp}{2(p+2)}} \|J(1+t)w(1)\|_{L_x^2}^{\frac{dp}{2(p+2)}} \\
& \lesssim_u \langle t\rangle^{-\frac{dp}{2(p+2)}} \|J(t)u(t)\|_{L_x^2}^{\frac{dp}{2(p+2)}}.
\end{align*}
Thus, using $\tfrac{4}{d}< p< \tfrac{4}{d-2}$ and the usual Sobolev embedding (plus unitarity of the free propagator), we obtain
\begin{align*}
\|w(1)\|_{L_x^{p+2}}^{p+2} & \lesssim \|w(1)\|_{L_x^{p+2}}^{\frac{4(p+2)}{dp}}\|w(1)\|_{L_x^{p+2}}^{p+2-\frac{4(p+2)}{dp}}\\
& \lesssim_u \langle t\rangle^{-2}\|Ju\|_{L_x^2}^2 \|u\|_{H_x^1}^{p+2-\frac{4(p+2)}{dp}} \\
& \lesssim_u \langle t\rangle^{-2}e(t).
\end{align*}

Altogether, we obtain a differential inequality of the form 
\[
\tfrac{d}{dt}e(t) \leq f(t)e(t) + C\langle t\rangle^{-2}e(t),
\]
where
\[
f(t) = \begin{cases} 0 & p>\frac{8}{d} \\ \frac{4-\frac{dp}{2}}{t} & \tfrac{4}{d}<p\leq \tfrac{8}{d}. \end{cases}  
\]
Thus by Gronwall's inequality,
\[
\|Ju\|_{L_x^2}^2 + t^2\|w\|_{L_{\sigma,x}^{p+2}}^{p+2} \lesssim \begin{cases} 1 & p>\frac{8}{d} \\ t^{4-\frac{dp}{2}} & \tfrac{4}{d}<p\leq \frac{8}{d}\end{cases}
\]
for all $t\geq 1$.  Combining these with the bounds on $[0,1]$ yields the estimates appearing the statement of the proposition.\end{proof}

We now leverage the decay provided by Proposition~\ref{P:decay} to obtain critical space-time bounds for the solution.  We note that while the estimate obtained in Proposition~\ref{P:decay} for $Ju$ is worse than the corresponding estimate one obtains for the standard NLS (by a power of $t$ in the regime $\tfrac{4}{d}<p\leq\tfrac{8}{d}$), we can still gain quite a bit from this estimate by combining it with the Klainerman--Sobolev estimate and a Strichartz estimate (with respect to the $(\sigma,x)$ variables). Indeed, while the bound in Proposition~\ref{P:decay} degenerates to the trivial bound $\|Ju\|_{L^2}\lesssim \langle t\rangle$ as $p\downarrow\tfrac{4}{d}$, we will see that (in dimensions $d\geq 2$) we can use this estimate to obtain critical space-time bounds of the form $e^{i\sigma\Delta}u\in L_{\sigma,t}^{q_c} L_x^{r_c}$ with $q_c<\infty$ for the full mass-supercritical range.  (The argument will show that we must take $q_c\to\infty$ as $p\downarrow \tfrac{4}{d}.$) 

%

In the following, we define the exponent
\[
p_0 = 3+\sqrt{5}\approx 5.2361
\]
In the case of dimension $d=1$, we will need to impose $p>p_0$.

\begin{proposition}[Critical bounds]\label{P:bounds} 

Let $\tfrac{4}{d}<p\leq\tfrac{4}{d-2}$ in dimensions $d\geq 3$, with $2<p<\infty$ in $d=2$.  There exists 
\[
q_c\in(p+1,\infty)
\]
and $r_c$ satisfying
\begin{equation}\label{critical-line}
\tfrac{2}{q_c}+\tfrac{d}{r_c}=\tfrac{2}{p}
\end{equation}
such that
\[
w=e^{i\sigma\Delta}u\in L_{\sigma,t}^{q_c}L_x^{r_c}([0,1]\times\R\times\R^d). 
\]

Furthermore, there exists a non-endpoint admissible Strichartz pair $(q,r)$ such that
\begin{equation}\label{companion-scaling}
\tfrac{p}{q_c}+\tfrac{2}{q} = 1 \qtq{and} \tfrac{p}{r_c} + \tfrac{2}{r} = 1. 
\end{equation} 

In dimension $d=1$ and $p_0<p<\infty$ we have
\[
w=e^{i\sigma\Delta}\in L_t^{2p} L_\sigma^\infty L_x^p([\R\times[0,1]\times\R).
\]
\end{proposition}

\begin{proof} 

Throughout the proof, we use the notation 
\[
s_c=\tfrac{d}{2}-\tfrac{2}{p}\in(0,1].
\]

We begin with the case $d\geq 2$.  We first consider arbitrary $q_c\in(\max\{p,2\},\infty)$ and define $r_c\in(\tfrac{dp}{2},\infty)$ via 
\[
\tfrac{2}{q_c}+\tfrac{d}{r_c}=\tfrac{2}{p}.
\]
We first verify the short-time estimate 
\begin{equation}\label{shorttime}
e^{i\sigma\Delta}u \in L_{\sigma,t}^{q_c} L_x^{r_c}([0,1]\times I \times\R^d)\qtq{for any finite}I. 
\end{equation}
To this end, we define $\rho\in[2,\infty)$ by 
\[
\tfrac{d}{r_c}=\tfrac{d}{\rho}-s_c.
\]
In particular, combining the scaling relations, we find that $(q_c,\rho)$ is a Schr\"odinger admissible pair. Thus by Sobolev embedding, the Strichartz estimate, and H\"older's inequality, we obtain
\begin{align*}
\|e^{i\sigma\Delta}u(t)\|_{L_{t,\sigma}^{q_c} L_x^{r_c}(I\times[0,1]\times\R^d)} & \lesssim \|e^{i\sigma\Delta}|\nabla|^{s_c} u\|_{L_{t,\sigma}^{q_c} L_x^\rho(I\times[0,1]\times\R^d)} \\
& \lesssim \| |\nabla|^{s_c} u(t)\|_{L_t^{q_c} L_x^2(I\times\R^d)} \\
& \lesssim |I|^{\frac{1}{q_c}} \| |\nabla|^{s_c} u(t)\|_{L_t^\infty L_x^2},
\end{align*}
which yields \eqref{shorttime}.

It remains to show that for suitable $(q_c,r_c)$ (in particular, for sufficiently large $q_c$), we have the long-time estimate
\begin{equation}\label{longtime}
e^{i\sigma\Delta}u \in L_{\sigma,t}^{q_c} L_x^{r_c}([0,1]\times\{|t|\geq 2\}\times\R^d).
\end{equation}
We impose $|t|\geq2$ so that $|t+\sigma|\geq \tfrac12|t|$ uniformly in $\sigma\in[0,1]$. 

In what follows, we let
\[
c_1 = \begin{cases} 0 & p>\tfrac{8}{d} \\ 2-\tfrac{dp}{4} & \tfrac{4}{d}<p\leq \tfrac{8}{d}.\end{cases}
\]
so that the result of Proposition~\ref{P:decay} may be written
\begin{equation}\label{Prop52bds}
\|J(t)u(t)\|_{L_x^2} \lesssim \langle t\rangle^{c_1}.
\end{equation}
We also observe that for $s\in[0,1]$, 
\[
\|J^s(t) u(t)\|_{L_x^2} \lesssim \|u(t)\|_{L_x^2}^{1-s} \|J(t)u(t)\|_{L_x^2}^s \lesssim \langle t\rangle^{sc_1}. 
\]
Indeed, this follows from the representation $J^s(t) = e^{i\frac{|x|^2}{4t}}(-4t^2\Delta)^{\frac{s}{2}} e^{-i\frac{|x|^2}{4t}}$ and the interpolation estimate
\[
\| |\nabla|^s u\|_{L_x^2} \lesssim \|u\|_{L_x^2}^{1-s} \|\nabla u\|_{L_x^2}^s. 
\]

Defining $\rho$ as above, we now use the Klainerman--Sobolev embedding and Strichartz estimate to obtain (for $(t,\sigma,x)\in\{|t|\geq 2\}\times[0,1]\times\R^d$):
\begin{align*}
\|e^{i\sigma\Delta}u(t)\|_{L_{t,\sigma}^{q_c} L_x^{r_c}} & \lesssim \| |t+\sigma|^{-s_c} \| [-4(t+\sigma)^2\Delta]^{\frac{s_c}{2}} \overline{M(t+\sigma)} e^{i\sigma\Delta} u(t)\|_{L_\sigma^{q_c} L_x^{\rho}} \|_{L_{t}^{q_c} } \\
& \lesssim \| \langle t\rangle^{-s_c} \|J^{s_c}(t+\sigma) e^{i\sigma\Delta}u(t)\|_{L_\sigma^{q_c} L_x^{\rho}}\|_{L_{t}^{q_c}} \\
& \lesssim \|\langle t\rangle^{-s_c} \|e^{i\sigma\Delta}J^{s_c}(t)u(t)\|_{L_{\sigma}^{q_c} L_x^{\rho}} \|_{L_t^{q_c}} \\
& \lesssim \|\langle t\rangle^{-s_c} \|J^{s_c}u(t)\|_{L_x^2} \|_{L_t^{q_c}} \\
& \lesssim \| \langle t\rangle^{-s_c[1-c_1]}\|_{L_t^{q_c}}. 
\end{align*}

Recalling the definition of $c_1$ above, we deduce that
\[
u\in L_{\sigma,t}^{q_c} L_x^{r_c}([0,1]\times\R\times\R^d) \qtq{for} q_c> Q(d,p):=\begin{cases} \tfrac{2p}{dp-4} & p>\tfrac{8}{d} \\ \tfrac{8p}{(dp-4)^2} & \tfrac{4}{d}<p\leq \tfrac{8}{d}.\end{cases}
\]

We now consider the construction of the `companion' non-endpoint Strichartz pair $(q,r)$ satisfying \eqref{companion-scaling}.  We proceed as follows:  Fix $q_c\in (\max\{p+1,Q(d,p)\},\infty)$ and define $q$ via the first identity in \eqref{companion-scaling}.  In particular, we have $q\in(2,\infty)$.  We then choose $r$ such that $\tfrac{2}{q}+\tfrac{d}{r}=\tfrac{d}{2}$, which also implies the second identity in \eqref{companion-scaling}. For dimensions $d\geq 2$, we can therefore find the desired pairs $(q_c,r_c)$ and $(q,r)$.

The approach described above may also be applied when $d=1$, but seems to only work for $p\geq 6$ (in which case the time exponent of the companion Strichartz pair reaches the $L_t^4$ endpoint).  We present a modification that can cover the range $p>p_0$, where we recall $p_0 = 3+\sqrt{5}\approx 5.2361$.  In particular, it seems that we cannot treat the full range $p>4$ via this approach. 

We will show $e^{i\sigma\Delta}u\in L_\sigma^\infty L_t^{2p}L_x^p$.  For the short-time estimate, we use Minkowski's inequality and Sobolev embedding to obtain 
\begin{align*}
\|e^{i\sigma\Delta}u(t)\|_{L_t^{2p} L_\sigma^\infty L_x^{p}(I\times [0,1] \times \R)} & \lesssim \| \| e^{i\sigma\Delta} |\nabla|^{\frac12-\frac1p} u(t)\|_{L_\sigma^\infty L_x^2} \|_{L_t^{2p}} \\
& \lesssim \|u\|_{L_t^{2p} H_x^1} \lesssim |I|^{\frac{1}{2p}} \|u\|_{L_t^\infty H_x^1}. 
\end{align*}

For the long-time estimate (on $\{|t|\geq 2\}$), we use the Klainerman--Sobolev embedding to obtain
\begin{align*}
\|e^{i\sigma\Delta}u(t)\|_{ L_t^{2p} L_\sigma^\infty L_x^p} & \lesssim \| |t+\sigma|^{-[\frac12-\frac1p]} \|J^{\frac12-\frac1p}(t+\sigma)e^{i\sigma\Delta} u(t)\|_{L_\sigma^\infty L_x^2} \|_{L_t^{2p}} \\
& \lesssim \| \langle t\rangle^{-[\frac12-\frac1p]} \|J^{\frac12-\frac1p}(t)u(t)\|_{L_x^2}\|_{L_t^{2p}} \\
& \lesssim \|\langle t\rangle^{-[\frac12 - \frac1p][1-c_1]}\|_{L_t^{2p}}. 
\end{align*}
This norm is clearly finite for $p>8$, while for $4\leq p <8$ this requires
\[
(p-2)\cdot\tfrac{p-4}{4}>1,\qtq{i.e.} p>3+\sqrt{5}. 
\]

Thus for $p>p_0$ we obtain $e^{i\sigma\Delta}u\in L_t^{2p} L_\sigma^{\infty} L_x^p$, as desired. \end{proof}

It remains to use the critical bounds obtained in Proposition~\ref{P:bounds} to prove scattering.

\begin{proposition}[Scattering]\label{P:scatter} For $\tfrac{4}{d}<p\leq \tfrac{4}{d-2}$ (with $p_0<p<\infty$ in dimension $d=1$ and $2<p<\infty$ in dimension $d=2$) and $u_0\in \Sigma$, the solution $u$ to \eqref{nls} with $u|_{t=0}=u_0$ scatters in $\Sigma$.  That is, there exist $u_\pm\in\Sigma$ such that
\[
\lim_{t\to\pm\infty}\|e^{-it\Delta}u(t) - u_\pm\|_{\Sigma}=0.
\]
\end{proposition}

\begin{proof} We continue using the notation from above, e.g. $w(\sigma)=U(\sigma)u$. 

We first consider the case $d\geq 2$ and let $(q_c,r_c)$ and $(q,r)$ be the exponent pairs obtained in Proposition~\ref{P:bounds}.  We begin by proving that
\[
w(\theta)\in L_\theta^{q_c} L_t^q L_x^r([0,1]\times\R\times\R^d).
\]
Without loss of generality, we focus on estimating on the domain $[0,1]\times[0,\infty)\times\R^d$. 

We let $\eta>0$ be a small parameter (to be determined more precisely below) and split $[0,\infty)$ into finitely many intervals $I_j=[t_j,t_{j+1}]$ such that 
\[
\|w(\theta)\|_{L_t^{q_c} L_\theta^{q_c} L_x^{r_c}(I_j\times[0,1]\times\R^d)} < \eta \qtq{for each}j. 
\]
We will show that
\begin{equation}\label{L2-based-bootstrap}
\|w(\theta)\|_{L_\theta^{q_c}L_t^q L_x^r([0,1]\times I_j\times\R^d)} \lesssim \|u(t_j)\|_{L_x^2} \lesssim_u 1 \qtq{for each} j,
\end{equation}
so that by the triangle inequality we derive $w\in L_\theta^{q_c} L_t^q L_x^r([0,1]\times[0,\infty)\times\R^d)$, as desired.  Here the uniform bound on the $L^2$-norm of $u$ is provided by mass conservation.

We have (with all estimates on time interval $I_j$):
\begin{align*}
\|e^{i\theta\Delta}u\|_{L_\theta^{q_c} L_t^q L_x^r} & \lesssim \|e^{i\theta\Delta}e^{i(t-t_j)\Delta}u(t_j)\|_{L_\theta^{q_c} L_t^q L_x^r} \\
& \quad + \biggl\| \int_0^t \int_0^1 e^{i(t-s+\theta-\sigma)\Delta}[|e^{i\sigma\Delta} u|^p e^{i\sigma\Delta} u]\,d\sigma\,ds\biggr\|_{L_\theta^{q_c} L_t^q L_x^r}.
\end{align*} 

For the linear term, we change variables in $t$ and apply Strichartz estimates to obtain a bound of
\begin{align*}
\|e^{i(t-t_j)\Delta}u(t_j)\|_{L_\theta^{q_c} L_t^q L_x^r} \lesssim \|e^{i(t-t_j)\Delta}u(t_j)\|_{L_t^q L_x^r} \lesssim \|u(t_j)\|_{L_x^2}.
\end{align*}
Using the shifted Strichartz estimate and H\"older's inequality (recalling the scaling relations \eqref{companion-scaling} and the fact that $q_c>p+1$), we next estimate the nonlinear term by
\begin{align*}
\int_0^1&\biggl\| \int_{t_j}^t  e^{i(t-s+\theta-\sigma)\Delta}[|e^{i\sigma\Delta} u|^p e^{i\sigma\Delta}u]\,ds\biggr\|_{L_\theta^{q_c} L_t^q L_x^r} \,d\sigma \\
& \lesssim \int_0^1 \| |e^{i\sigma\Delta} u|^p e^{i\sigma\Delta} u\|_{L_\theta^{q_c} L_t^{q'} L_x^{r'}}\,d\sigma \\
& \lesssim \int_0^1  \|e^{i\sigma\Delta}u\|_{L_t^{q_c}L_x^{r_c}}^p \|e^{i\sigma\Delta}u\|_{L_t^q L_x^r} \,d\sigma \\
& \lesssim \| e^{i\sigma\Delta}u\|_{L_{\sigma,t}^{q_c} L_x^{r_c}}^p \|e^{i\sigma\Delta} u\|_{L_\sigma^{q_c} L_t^q L_x^r}  \lesssim \eta^p \|e^{i\sigma\Delta}u\|_{L_\sigma^{q_c}L_t^q L_x^r}.
\end{align*}
Thus we can absorb this term into the left-hand side and thereby obtain \eqref{L2-based-bootstrap}.

A similar argument, using the uniform boundedness of $\nabla u$ in $L^2$ (a consequence of energy conservation), the chain rule, and the fact that $\nabla$ commutes with the free propagator, implies that
\[
\nabla w(\theta)\in L_\theta^{q_c} L_t^q L_x^r([0,1]\times\R\times\R^d).
\]

We now wish to prove that 
\[
e^{i\theta\Delta}J(t)u(t) = J(t+\theta)w(\theta)\in L_\theta^{q_c} L_t^q L_x^r([0,1]\times\R\times\R^d).
\]
This requires a slightly modified argument, as we do not yet know that $J(t)u(t)\in L_t^\infty L_x^2$ for the full range of $p$ under consideration. Restricting again to an interval $I_j=[t_j,t_{j+1}]$, we wish to prove that
\begin{equation}\label{Ju-bootstrap}
\|Ju\|_{L_t^\infty L_x^2} + \|e^{i\theta\Delta}Ju\|_{L_\theta^{q_c}L_t^q L_x^r} \lesssim \|J(t_j)u(t_j)\|_{L_x^2} + \eta^p \|e^{i\theta\Delta}Ju\|_{L_\theta^{q_c} L_t^q L_x^r}.
\end{equation}
Indeed, we can then absorb the last term on the right-hand side back into the left-hand side and obtain 
\[
\|Ju\|_{L_t^\infty L_x^2} + \|e^{i\theta\Delta}Ju\|_{L_\theta^{q_c}L_t^q L_x^r} \lesssim \|J(t_j)u(t_j)\|_{L_x^2}. 
\]
This shows that the $L^2$-norm of $Ju$ can at most double (say) on each $I_j$, so that in particular $Ju$ is bounded in $L_t^\infty L_x^2$ globally in time. Returning to \eqref{Ju-bootstrap}, this then gives uniform control of $e^{i\theta\Delta}Ju$ in $L_\theta^{q_c}L_t^q L_x^r$ on each $I_j$, and hence globally in time.  Thus it remains to establish \eqref{Ju-bootstrap}. 

We begin with the $L_\theta^{q_c} L_t^q L_x^r$ estimate. The contribution of the linear term is bounded by
\begin{align*}
\| e^{i\theta\Delta}J(t)e^{i(t-t_j)\Delta}u(t_j)\|_{L_\theta^{q_c} L_t^q L_x^r} & = \|e^{i[\theta+t-t_j]\Delta} J(t_j)u(t_j)\|_{L_\theta^{q_c}L_t^q L_x^r} \\
& \lesssim \|e^{i(t-t_j)\Delta} J(t_j)u(t_j)\|_{L_\theta^{q_c} L_t^q L_x^r([0,1]\times\R\times\R^d)} \\
& \lesssim \|J(t_j)u(t_j)\|_{L_x^2}. 
\end{align*} 
For the nonlinear term, we estimate using the shifted Strichartz estimate, commutation properties of $J$, and the chain rule for $J$: 
\begin{align*}
\biggl\| & e^{i\theta\Delta}J(t)\int_{t_j}^t\int_0^1 e^{i(t-s-\sigma)\Delta}[|e^{i\sigma\Delta}u|^p e^{i\sigma\Delta}u]\,ds\,d\sigma\biggr\|_{L_\theta^{q_c} L_t^q L_x^r} \\
& \lesssim \int_0^1 \biggl\| \int_{t_j}^t e^{i(t-s+\theta-\sigma)\Delta}J(s+\sigma)[|e^{i\sigma\Delta}u|^p e^{i\sigma\Delta}u]\,ds\biggr\|_{L_\theta^{q_c} L_t^q L_x^r} \,d\sigma \\
& \lesssim \int_0^1 \| J(t+\sigma)[|e^{i\sigma\Delta}u|^p e^{i\sigma\Delta}u]\|_{L_\theta^{q_c} L_t^{q'}L_x^{r'}} \,d\sigma \\
& \lesssim \int_0^1 \|e^{i\sigma\Delta}u\|_{L_t^{q_c} L_x^{r_c}}^p \| J(t+\sigma)e^{i\sigma\Delta}u\|_{L_t^q L_x^r}\,d\sigma \\
& \lesssim \|e^{i\sigma\Delta}u\|_{L_\sigma^{q_c}L_t^{q_c}L_x^{r_c}}^p \|e^{i\sigma\Delta}Ju\|_{L_\sigma^{q_c}L_t^q L_x^{r}} \lesssim \eta^p\|e^{i\sigma\Delta}Ju\|_{L_\sigma^{q_c}L_t^q L_x^{r}} 
\end{align*}
as desired. 

We turn to the $L_t^\infty L_x^2$ estimate. We fix $t\in[t_j,t_{j+1}]$. The contribution of the linear term is bounded by
\[
\|J(t)e^{i(t-t_j)\Delta}u(t_j)\|_{L_x^2} = \|e^{it\Delta} J(t_j)u(t_j)\|_{L_x^2} = \|J(t_j)u(t_j)\|_{L_x^2}.
\]
For the nonlinear term, we estimate similarly to the above:
\begin{align*}
\biggl\| &J(t)\int_{t_j}^t\int_0^1 e^{i(t-s-\sigma)\Delta}[|e^{i\sigma\Delta}u|^p e^{i\sigma\Delta}u]\,d\sigma\,ds \biggr\|_{L_x^2} \\
& = \biggl\| \int_{t_j}^t\int_0^1 e^{i(t-s-\sigma)\Delta}J(s+\sigma)[|e^{i\sigma\Delta}u|^p e^{i\sigma\Delta}u]\,d\sigma\,ds \biggr\|_{L_x^2} \\
& \leq \int_0^1 \biggl\|\int_{t_j}^t e^{-is\Delta}J(s+\sigma)[|e^{i\sigma\Delta}u|^p e^{i\sigma\Delta}u]\,ds \biggr\|_{L_x^2} \,d\sigma \\
& \lesssim \int_0^1  \|J(t+\sigma)[|e^{i\sigma\Delta}u(t)|^p e^{i\sigma\Delta}u(t)]\|_{L_t^{q'}L_x^{r'}} \,d\sigma \\
& \lesssim \| e^{i\sigma\Delta}u\|_{L_\sigma^{q_c} L_t^{q_c} L_x^{r_c}}^p \|e^{i\sigma\Delta}Ju\|_{L_\sigma^{q_c} L_t^q L_x^r}  \lesssim \eta^p \|e^{i\sigma\Delta}Ju\|_{L_\sigma^{q_c} L_t^q L_x^r},
\end{align*}
as desired.  We conclude that \eqref{Ju-bootstrap} holds. 

At this point, we have established global bounds for $e^{i\theta\Delta}u$, $e^{i\theta\Delta}\nabla u$, and $e^{i\theta\Delta}Ju$ in $L_\theta^{q_c} L_t^q L_x^r$. We turn now to the proof of scattering.  Without loss of generality, we consider scattering forward in time. We will show $\{e^{-it\Delta}u(t)\}$ is Cauchy in $\Sigma$ as $t\to\infty$, beginning with the $L^2$-norm.

We estimate as above: for $t_1<t_2$, 
\begin{align*}
\|&e^{-it_2\Delta}u(t_2)-e^{-it_1\Delta}u(t_1)\|_{L_x^2} \\
& \lesssim \biggl\| \int_{t_1}^{t_2} \int_0^1 e^{-i(s+\sigma)\Delta}[|e^{i\sigma\Delta}u|^p e^{i\sigma\Delta}u]\,d\sigma\,ds\biggr\|_{L_x^2} \\
& \lesssim \int_0^1 \biggl\| \int_{t_1}^{t_2} e^{-is\Delta}[|e^{i\sigma\Delta}u|^p e^{i\sigma\Delta}u]\,ds\biggr\|_{L_x^2} \,d\sigma\\
& \lesssim \int_0^1 \|e^{i\sigma\Delta}u\|_{L_t^{q_c} L_x^{r_c}([t_1,\infty)\times\R^d)}^p \|e^{i\sigma\Delta}u\|_{L_t^q L_x^r} \,d\sigma \\
& \lesssim \|e^{i\sigma\Delta}u\|_{L_{t,\sigma}^{q_c} L_x^{r_c}([t_1,\infty)\times[0,1]\times\R^d)}^p \|e^{i\sigma\Delta} u\|_{L_\sigma^{q_c} L_t^q L_x^r}  \to 0 \qtq{as}t_1\to\infty. 
\end{align*}
Using the fact that $\nabla$ commutes with $e^{i\sigma\Delta}$ and the chain rule, we can similarly obtain that $\nabla e^{-it\Delta} u(t)$ is Cauchy in $L^2$ as $t\to\infty$.  Finally, using the commutation properties of $J(t)$ as above, we can also obtain that $xe^{-it\Delta}u(t)$ is Cauchy in $L^2$ as $t\to\infty$.  Indeed, we have by the estimates used above that
\begin{align*}
\|&xe^{-it_2\Delta}u(t_2) - xe^{-it_1\Delta}u(t_1)\|_{L_x^2} \\
& \lesssim \biggl\| \int_{t_1}^{t_2} xe^{-i(s+\sigma)\Delta}[|e^{i\sigma\Delta}u|^p e^{i\sigma\Delta}u]\,d\sigma\,ds\biggr\|_{L_x^2} \\
& \lesssim \int_0^1 \biggl\| \int_{t_1}^{t_2} e^{-is\Delta}J(s+\sigma)[|e^{i\sigma\Delta}u|^p e^{i\sigma\Delta}u]\,ds\biggr\|_{L_x^2}\,d\sigma \\
& \lesssim \|e^{i\sigma\Delta}u\|_{L_{t,\sigma}^{q_c} L_x^{r_c}([t_1,\infty)\times[0,1]\times\R^d)}^p \|e^{i\sigma\Delta}Ju\|_{L_\sigma^{q_c} L_t^q L_x^r}  \to 0 \qtq{as}t_1\to\infty,
\end{align*}
which finally completes the proof in the case $d\geq 2$. 

It remains to treat the case $d=1$, in which case we have the estimate $e^{i\sigma\Delta}u \in L_t^{2p} L_\sigma^\infty L_x^p$. The `companion' Strichartz pair in this case will be $L_t^4 L_x^\infty$.  Let us show, for example, how to obtain $e^{i\sigma\Delta} Ju\in L_\sigma^\infty L_t^4 L_x^\infty$.  This will demonstrate the basic nonlinear estimate, and then the proof of scattering in $\Sigma$ follows essentially exactly as above. 

As before, let us split into an interval $I=(t_0,t_1)$ where the $L_t^{2p} L_\sigma^{\infty} L_x^p$-norm of $e^{i\sigma\Delta}u$ is small.  (In particular, the $L_\sigma^\infty L_t^{2p} L_x^p$-norm is small, as well.) Using the chain rule for $J$ and commutation properties of $J$, along with the shifted Strichartz estimate, we obtain
\begin{align*}
\|&e^{i\theta\Delta} Ju\|_{L_\theta^\infty L_t^4 L_x^\infty (I)}&\\
& \lesssim \|J(t_0)u(t_0)\|_{L_x^2} + \biggl\| \int_0^t\int_0^1 e^{i(t+\theta-s-\sigma)\Delta}J(s+\sigma)[|w|^p w]\,d\sigma\biggr\|_{ L_\theta^\infty L_t^4 L_x^\infty(I)} \\
& \lesssim \|J(t_0)u(t_0)|\|_{L_x^2} + \int_0^1 \|J(t+\sigma)[|w|^p w]\|_{L_t^\frac43 L_x^1(I)}\,d\sigma \\
& \lesssim \|J(t_0)u(t_0)\|_{L_x^2} +\int_0^1 \|w\|_{L_t^{2p} L_x^{p}(I)}^p \|J(t+\sigma)e^{i\sigma\Delta}u(t)\|_{L_t^4 L_x^\infty(I)}\,d\sigma \\
& \lesssim \|J(t_0)u(t_0)\|_{L_x^2} + \|e^{i\sigma\Delta}u\|_{L_\sigma^\infty L_t^{2p} L_x^p(I)}^p \|e^{i\sigma\Delta} Ju\|_{L_\sigma^\infty L_{t}^4 L_x^\infty}. \\
\end{align*}
Thus, the $L_\sigma^\infty L_t^4 L_x^\infty$-norm of $e^{i\sigma\Delta}Ju$ can at most double (say) on any such interval, and hence (as above) we derive global $L_\sigma^\infty L_t^4 L_x^\infty$-bounds for $e^{i\sigma\Delta}Ju$.  Estimating as above, we finally obtain scattering in $\Sigma$ in the case $d=1$ as well.  In particular, to prove that $e^{-it\Delta}u(t)$ is Cauchy in $H^{1,1}$ will ultimately rely on the fact that
\[
\lim_{t_1\to\infty} \|e^{i\sigma\Delta}u\|_{L_t^{2p} L_\sigma^\infty L_x^p((t_1,\infty)\times[0,1]\times\R)} = 0,
\]
which follows from the dominated convergence theorem. \end{proof}

\section{Failure of scattering for sufficiently low powers}\label{S:T4}

 In this section we prove Theorem~\ref{T4} (reproduced here as Theorem~\ref{T4_informal}), showing that (unmodified) scattering must fail if the power of the nonlinearity is too low.  We adapt the argument of \cite{Glassey, Strauss}.

\begin{theorem}\label{T4_informal} Let $0<p\leq 1$ and $p\leq \tfrac{2}{d}$.  Suppose $u$ is a forward-global solution to \eqref{nls} such that
\[
\|u(t)-e^{it\Delta}\varphi\|_{L_x^2} = 0 \qtq{as}t\to\infty
\]
for some $\varphi\in L^2$.  Then $\varphi\equiv 0$. 
\end{theorem}

\begin{proof}
We suppose towards a contradiction that $\varphi\neq 0$.  We define $v(t)=e^{it\Delta}\varphi$ and choose $\psi\in\mathcal{S}$ satisfying
\[
C_0:=\langle |\hat \varphi|^p \hat\varphi,\hat\psi\rangle  >0.
\]
Defining $z(t)=e^{it\Delta}\psi$, direct computation yields
\begin{align}
i\partial_t \langle u(t),z(t)\rangle & = \int_0^1 \langle |U(\sigma)u(t)|^p U(\sigma)u(t),U(\sigma)z(t)\rangle\,d\sigma \nonumber \\ 
&=  \int_0^1 \langle |U(\sigma)v(t)|^p U(\sigma)v(t),U(\sigma)z(t)\rangle\,d\sigma \label{GSB-main}\\
& \ + \int_0^1 \langle |U(\sigma)u(t)|^p U(\sigma)u(t)-|U(\sigma)v(t)|^p U(\sigma)v(t),U(\sigma)z(t)\rangle\,d\sigma. \label{GSB-error}
\end{align}

The term \eqref{GSB-error} is treated as an error term. By H\"older's inequality, the dispersive estimate, and the fact that $\psi\in\mathcal{S}$,
\begin{align*}
|\eqref{GSB-error}| &\leq \int_0^1 \| |U(\sigma)u|^p U(\sigma)u - |U(\sigma) v|^p U(\sigma)v\|_{L_x^{\frac{2}{p+1}}} \|U(\sigma+t)\psi\|_{L_x^{\frac{2}{1-p}}}\,d\sigma \\
& \lesssim_\psi \int_0^1 [\|U(\sigma)u\|_{L_x^2}^p +\|U(\sigma)v\|_{L_x^2}^p]\|U(\sigma)[u-v]\|_{L_x^2} |\sigma+t|^{-\frac{dp}{2}}\,d\sigma \\
& \lesssim_\psi \int_0^1 \|\varphi\|_{L_x^2}^p \|u(t)-v(t)\|_{L_x^2} |t|^{-\frac{dp}{2}}\,d\sigma \\
& \lesssim_{\varphi,\psi} |t|^{-\frac{dp}{2}}\|u(t)-v(t)\|_{L_x^2} \\
& = o(t^{-\frac{dp}{2}})\qtq{as}t\to\infty. 
\end{align*}

We now use the factorization
\[
U(t)=M(t)D(t)\F M(t)
\]
(see \eqref{dollard} for this notation) to extract the leading order term in \eqref{GSB-main}.  We begin by writing
\begin{align*}
\eqref{GSB-main} & = \int_0^1 \langle |U(t+\sigma)\varphi|^p U(t+\sigma)\varphi,U(t+\sigma)\psi\rangle\,d\sigma \\
& = \int_0^1 [2(t+\sigma)]^{-\frac{dp}{2}} \langle |\F M(t+\sigma)\varphi|^p \F M(t+\sigma)\varphi,\F M(t+\sigma)\psi\rangle\,d\sigma. 
\end{align*}

We now claim that for any $\eps>0$, there exists $T$ sufficiently large such that for any $s\geq T$, we have
\begin{align*}
|\langle |\F M(s)\varphi|^p \F M(s)\varphi, \F M(s)\psi\rangle - \langle |\hat\varphi|^p \hat \varphi,\hat\psi\rangle| < \eps. 
\end{align*}
To prove this we write
\begin{align}
|\langle & |\F M(s)\varphi|^p \F M(s)\varphi, \F M(s)\psi\rangle - \langle |\hat\varphi|^p \hat \varphi,\hat\psi\rangle| \nonumber \\
& \leq  |\langle |\F M(s)\varphi|^p \F M(s)\varphi - |\hat\varphi|^p \hat \varphi, \F M(s)\psi\rangle | \label{FM1}\\
& \quad + |\langle |\hat\varphi|^p \hat\varphi, \F[M(s)-1]\psi\rangle |\label{FM2}
\end{align}
 The first term is estimated by H\"older's inequality, the Plancherel Theorem, the Hausdorff--Young inequality, and the dominated convergence theorem:
 \begin{align*}
 |\eqref{FM1}|&\lesssim \| [|\F M\varphi|^p + |\hat\varphi|^p]\,|\F[M-1]\varphi|\,|\F M(s)\psi| \|_{L_x^1} \\
 & \lesssim \|\varphi\|_{L_x^2}^p\| \F[M-1]\varphi\|_{L_x^2} \|\F M(s)\psi\|_{L_x^{\frac{2}{1-p}}} \\
 & \lesssim \|\varphi\|_{L_x^2}^p \|\psi\|_{L_x^{\frac{2}{p+1}}} \|[M(s)-1]\varphi\|_{L_x^2} \\
 & \to 0 \qtq{as}s\to\infty.
 \end{align*}
 The second term is similarly estimated by
 \begin{align*}
|\eqref{FM2}|& \lesssim \|\hat\varphi\|_{L_x^2}^{p+1} \|\F[M(s)-1]\psi\|_{L_x^{\frac{2}{1-p}}} \\
& \lesssim \|\varphi\|_{L_x^2}^{p+1} \|[M(s)-1]\psi\|_{L_x^{\frac{2}{1+p}}} \\
& \to 0 \qtq{as}s\to\infty. 
 \end{align*}
 Thus, continuing from \eqref{GSB-main} and recalling the definition of $C_0$, we find that
 \[
 \eqref{GSB-main}  = C_0\int_0^1 [2(t+\sigma)]^{-\frac{dp}{2}}\,d\sigma + o(t^{-\frac{dp}{2}}) \qtq{as}t\to\infty. 
 \]
 
Combining the estimates above, we find that
 \[
 \partial_t \Im\langle u(t),z(t)\rangle \geq \tfrac12 C_0 t^{-\frac{dp}{2}}
 \]
 for all $t$ sufficiently large. To complete the proof, we observe that since $\tfrac{dp}{2}\leq1$, this lower bound contradicts the fact that
 \[
 |\langle u(t),z(t)\rangle| \lesssim \|\varphi\|_{L_x^2}\|\psi\|_{L_x^2}\lesssim 1
 \]
 uniformly in $t$. \end{proof}

 As mentioned in the introduction, in the case $d=1$ and $1<p\leq 2$ we also expect that unmodified scattering should fail.  In the case of the usual power-type NLS, these cases were addressed by \cite{Barab}, which utilized the pseudoconformal energy estimate.  As we failed to obtain any decay from the pseudoconformal energy estimate in the range $p<\tfrac{4}{d}$, the case $d=1$ and $1<p\leq 2$ remains open for \eqref{nls}.

\end{document}